\documentclass[12pt,oneside,english]{amsart}
\usepackage{geometry}
\usepackage{units}
\usepackage{mathrsfs}
\usepackage{amstext}
\usepackage{amsthm}
\usepackage{amssymb}
\usepackage{setspace}
\PassOptionsToPackage{normalem}{ulem}
\usepackage{ulem}
\makeatletter
\numberwithin{equation}{section}
\numberwithin{figure}{section}
\theoremstyle{plain}
\newtheorem{thm}{\protect\theoremname}[section]
\theoremstyle{remark}
\newtheorem{claim}[thm]{\protect\claimname}
\theoremstyle{plain}
\newtheorem{lem}[thm]{\protect\lemmaname}

\date{\today}

\makeatother

\usepackage{babel}
\providecommand{\claimname}{Claim}
\providecommand{\lemmaname}{Lemma}
\providecommand{\theoremname}{Theorem}

\begin{document}
\title{Shrinking Targets for Non-autonomous Systems}
\author{Marco Antonio L\'opez}
\begin{abstract}
In the present work we establish a Bowen-type formula for the Hausdorff
dimension of shrinking-target sets for non-autonomous conformal iterated
function systems in arbitrary dimensions and satisfying certain conditions.
In the case of dimension 1 we also investigate non-linear perturbations
of linear systems and obtain sufficient conditions under which the
perturbed systems satisfy the conditions in our hypotheses.
\end{abstract}

\maketitle

\section{Introduction}

In \cite{bowen_hausdorff_1979}, Rufus Bowen proved a dimension result
for certain dynamically-defined sets in terms of a topological pressure
function. Such formulas relating Hausdorff dimension to topological
pressure came to be known as Bowen's formula. Since Bowen's original
work, many others have extended Bowen's formula to several different
contexts \cite{baranski_bowens_2012,mauldin_dimensions_1996,mauldin_graph_2003,rugh_dimensions_2008}.
For a first introduction to Bowen's equation see Chapter 9 of \cite{barreira_ergodic_2012}.
The first results of such type for shrinking-target sets appeared
in a series of papers by Hill and Velani. In \cite{hill_metric_1997}
they prove that if $T$ is an expanding rational map of the Riemann
sphere with Julia set $J$, then for every $z_{0}\in J$ and $\tau>0$,
the Hausdorff dimension of the set
\[
W\left(\tau\right)=\bigcap_{m\ge1}\bigcup_{n\ge m}\left\{ z\in J\mid T^{n}\left(z\right)\in B\left(z_{0},\,\left|\left(T^{n}\right)'\left(z\right)\right|^{-\tau}\right)\right\} 
\]
is the unique solution to the equation $P\left(s\right)=0$, where
$P$ is a pressure function associated to the map $T$ and the constant
$\tau$ \cite{hill_metric_1997}.

In classical Diophantine approximation the set of $\alpha$-well approximable
numbers are 
\[
\mathscr{D}_{\alpha}=\bigcap_{n\ge1}\bigcup_{q\ge n}\bigcup_{0\le p\le q}\left\{ \theta\in\left[0,\,1\right]\backslash\mathbb{Q}\colon\left|\theta-\frac{p}{q}\right|\le\frac{1}{q^{\alpha}}\right\} .
\]
It is well known that if $0\leq\alpha\le2$ this set is $\left[0,\,1\right]\backslash\mathbb{Q}$.
By the Borel-Cantelli lemma, $\mathscr{D}_{\alpha}$ is a set of Lebesgue
measure zero for all $\alpha>2$. Thus, for such sets a natural question
is the Hausdorff dimension of $\mathscr{D}_{\alpha}.$ Jarnik \cite{jarnik_diophantische_1929}
and Besicovitch \cite{besicovitch_sets_1934} both proved that the
Hausdorff dimension of $\mathscr{D}_{\alpha}$ is $\frac{2}{\alpha}.$

The sets $W\left(\tau\right)$ and $\mathscr{D}_{\alpha}$ are examples
of shrinking-target sets. In dynamical systems and metric Diophantine
approximation shrinking-target sets have been studied in various contexts.
Two questions that often arise from shrinking target problems are
dichotomy laws or Borel-Cantelli lemmas (see \cite{chernov_dynamical_2001}
or \cite{sun_dichotomy_2017} for example), and Hausdorff dimension
of such sets. In this paper we will focus on the latter.

\subsection{Nonautonomous IFS}

Recently, Rempe-Gillen and Urba\'{n}ski \cite{rempe-gillen_non-autonomous_2015}
expanded Bowen's formula into the realm of nonautonomous iterated
function systems (IFSs).

An autonomous IFS consists of a countable indexing set $I$ called
the alphabet, and a collection $\left(\varphi_{a}\right)_{a\in I}$
of contracting maps on some set $X\subseteq\mathbb{R}^{d}.$ The Cartesian
product $I^{n}$ is referred to as the set of words of length $n$,
and for every $\omega\in I^{n}$ we define $\varphi_{\omega}^{n}\colon X\to X$
by the composition 
\[
\varphi_{\omega}^{n}=\varphi_{\omega_{1}}\circ\varphi_{\omega_{2}}\circ\cdots\circ\varphi_{\omega_{n}},
\]
 where $\omega_{j}$ denotes the $j$-the coordinate of $\omega$.

As an example, consider the celebrated middle-third Cantor set $C.$
Let $I=\left\{ 0,2\right\} $ and for each $a\in I$ define $\varphi_{a}\colon C\to C$
as 
\[
\varphi_{a}\left(x\right)=\frac{x+a}{3}.
\]
Note that if we consider the alphabet $I=\left\{ 0,1,2\right\} $
and define $\varphi_{a}$ as above, then each map $\varphi_{a}$ corresponds
to one of the three inverse branches of the expanding map $T\left(x\right)=3x$$\mod1$
on $\left[0,1\right]$. More precisely,
\[
T\left(\varphi_{a}\left(x\right)\right)=x,x\neq1
\]
for all $a\in I$ and 
\[
T^{n}\left(\varphi_{\omega}^{n}\left(x\right)\right)=x,x\neq1
\]
for all $\omega\in I^{n}.$

Instead of only considering one alphabet $I,$ a nonautonomous IFS
considers a countable collection $\left(I^{\left(n\right)}\right)_{n\in\mathbb{N}}$
of such alphabets. For each $n$ there is again a collection $\left(\varphi_{a}^{\left(n\right)}\right)_{a\in I^{\left(n\right)}}$
of contractions on $X$. Letting $I^{n}$ denote the cartesian product
$I^{\left(1\right)}\times I^{\left(2\right)}\times\cdots\times I^{\left(n\right)}$
we define 
\[
\varphi_{\omega}^{n}=\varphi_{\omega_{1}}^{\left(1\right)}\circ\varphi_{\omega_{2}}^{\left(2\right)}\circ\cdots\circ\varphi_{\omega_{n}}^{\left(n\right)},
\]
where $\omega_{j}\in I^{\left(j\right)}.$

In \cite{rempe-gillen_non-autonomous_2015} the authors consider nonautonomous
conformal iterated function systems $\Phi$ on $X\subset\mathbb{R}^{n}$
and their associated limit set
\[
J=\bigcap_{n\ge1}\bigcup_{\omega\in I^{n}}\varphi_{\omega}\left(X\right).
\]
Under suitable assumptions on $\Phi$, Rempe-Gillen and Urba\'{n}ski
show a Bowen-type formula for the limit set, that is, 
\[
\text{HD}\left(J\right)=\sup\left\{ t\ge0\mid\underline{P}\left(t\right)\ge0\right\} ,
\]
where
\begin{equation}
\underline{P}_{J}\left(t\right)=\liminf_{n\to\infty}\frac{1}{n}\log\left(\sum_{\omega\in I^{n}}\left\Vert D\varphi_{\omega}^{n}\right\Vert ^{t}\right)\label{eq:RU pressure}
\end{equation}

In \cite{fishman_shrinking_2015}, the authors explore the shrinking
target problem for a certain class of nonautonomous systems. Specifically,
for a sequence $Q=\left(q_{n}\right)$ of integers no smaller than
2, define 
\[
T_{n}=x\mapsto q_{n}x\text{ (mod 1)}\colon\left[0,\,1\right]\to\left[0,\,1\right].
\]

This sequence of maps gives rise to a nonautonomous dynamical system
on $\left[0,\,1\right]$ whose orbits are defined by $T^{n}=T_{n}\circ\cdots\circ T_{1}.$ 

Given a sequence $\alpha=\left(\alpha_{n}\right)\in\left(0,\,\infty\right)^{\mathbb{N}}$
and letting $\alpha\left(n\right)=\alpha_{1}+\cdots+\alpha_{n}$,
the shrinking target associated to $Q$ and $\alpha$ is defined as
\[
\mathscr{D}_{Q}\left(\alpha\right)=\bigcap_{m\ge0}\bigcup_{n\ge m}\left\{ x\in\left[0,\,1\right]\mid\left|T^{n}\left(x\right)\right|\le e^{-\alpha\left(n\right)}\right\} ,
\]
where $\left|x\right|$ denotes distance to the nearest integer. The
pressure associated to $Q$ and $\alpha$ is given by
\[
P_{Q}\left(t\right)=\limsup_{n\to\infty}\frac{1}{n}\log\left(\left(q_{1}q_{2}\cdots q_{n}\right)^{1-t}e^{-t\alpha\left(n\right)}\right).
\]

Note that $\mathscr{D}_{Q}\left(\alpha\right)$ can be rewritten in
terms of a non-autonomous IFS. Inded, if we define $I^{\left(n\right)}=\left\{ 0,1,\ldots,q_{n}-1\right\} $
and for, $a\in I^{\left(n\right)}$, $\varphi_{a}^{\left(n\right)}\colon\left[0,1\right]\to\left[0,1\right]$
as 
\[
\varphi_{a}^{\left(n\right)}\left(x\right)=q_{n}^{-1}\left(x+a\right),
\]
then
\[
\mathscr{D}_{Q}\left(\alpha\right)=\bigcap_{m\ge0}\bigcup_{n\ge m}\bigcup_{\omega\in I^{n}}\varphi_{\omega}^{n}\left(\left[0,1\right]\right).
\]

The main result in \cite{fishman_shrinking_2015} is an extension
of Bowen's formula, namely that 
\[
\text{HD}\left(\mathscr{D}_{Q}\left(\alpha\right)\right)=\sup\left\{ t\ge0\mid P\left(t\right)\ge0\right\} .
\]

Our main results, Theorems \ref{thm:boundedwithgaps} and \ref{thm:unboundedwithgaps},
establish Bowen's formula for a certain class of IFS coming from those
considered in \cite{rempe-gillen_non-autonomous_2015} satisfying
certain natural conditions. This class generalizes those IFSs in \cite{fishman_shrinking_2015}
in two important ways:
\begin{enumerate}
\item We consider a certain class of IFS in higher dimensions; that is,
on subsets of $\mathbb{R}^{d},$$d\ge1.$
\item We relax the condition that for fixed $n$, the derivatives $D\varphi_{a}^{\left(n\right)}\left(x\right)$
remain constan over all $x\in X$ and over all $a\in I^{\left(n\right)}.$
\end{enumerate}

\subsection{Organization}

In Section \ref{sec:Definitions-and-Preliminaries} we establish our
notation and basic definitions. In Section \ref{sec:Upper-bound}
we prove an upper bound for the Hausdorff dimension of our sets of
interest. The proof is fairly elementary and general. Our main results
are in Section \ref{sec:Lower-Bound}. It begins by defining and describing
all the conditions necessary in the hypothesis of our theorems. Then
we state and prove the main theorems. In Section \ref{sec:Ahlfors-Measures}
we pay special attention to one of the conditions in our hypotheses:
the existance of Ahlfors measures. We prove sufficient conditions
for their existance. Finally, in Section 6 we focus on the case in
Euclidean dimension $d=1$ and investigate the ``rigidity'' of IFS
satisfying our conditions. We prove that IFSs preserve all the required
conditions under sufficiently small perturbations.

\section{\label{sec:Definitions-and-Preliminaries}Definitions and Preliminaries}

Let $X\subset\mathbb{R}^{d}$ be a compact, convex subset with nonempty
interior an let $V$ be a bounded, open, connected set containing
$X$. Consider a countable collection $\left(I^{\left(n\right)}\right)_{n\in\mathbb{N}}$
of finite alphabets which will be used to encode a nonautonomous iterated
function system (IFS) in the following way. For every $n\in\mathbb{N}$
and every $j\in I^{\left(n\right)}$ we fix conformal contractions
$\varphi_{j}^{\left(n\right)}\colon V\to V$ such that $\varphi_{j}^{\left(n\right)}\left(X\right)\subseteq X$;
that is, there exists $\theta>0$ such that for all $n\in\mathbb{N}$
and all $j\in I^{\left(n\right)}$ we have that 
\[
\overline{\kappa}_{\left(n\right)}:=\max_{j\in I^{\left(n\right)}}\left\Vert \left.D\varphi_{j}^{\left(n\right)}\right.\right\Vert \le e^{-\theta},
\]
 and $D\varphi_{j}^{\left(n\right)}$ is a similarity.

Letting 
\[
I^{n}=\prod_{k=1}^{n}I^{\left(k\right)},
\]
 we define for every $\omega=\left(\omega_{1},\ldots,\omega_{n}\right)\in I^{n}$
the map 
\[
\varphi_{\omega}^{n}=\varphi_{\omega{}_{1}}^{\left(1\right)}\circ\varphi_{\omega_{2}}^{\left(2\right)}\circ\cdots\circ\varphi_{\omega_{n}}^{\left(n\right)}.
\]
Furthermore, products of the form $I^{\left(k\right)}\times I^{\left(k+1\right)}\times\cdots\times I^{\left(n\right)}$
will be denoted by $I^{\left(k,n\right)}$, where $n$ may be infinity.
The set $I^{\left(1,\infty\right)}$ will simply be denoted by $I^{\infty}.$
We will also make use of the shift map $\sigma$, which takes a word
$\omega\in I^{n}$ to a word $\sigma^{k}\omega\in I^{\left(k+1,\,n\right)}$
where 
\[
\sigma^{k}\omega=\left(\omega_{k+1},\ldots,\omega_{n}\right)
\]
 for all $0\le k<n$. The empty word $\sigma^{n}\omega$ is used to
encode the identity map, i.e., $\varphi_{\sigma^{n}\omega}=\text{id}$
for all $\omega\in I^{n}$. 

On the other hand, for $\omega\in I^{\left(m,\,n\right)}$ we will
let $\omega\vert_{k}$ denote the word $\left(\omega_{1},\ldots,\omega_{k}\right)\in I^{\left(m,m+k-1\right)}$
for all $1\le k\le n-m+1.$ 

To define a shrinking target set we fix a sequence $\left(\beta_{n}\right)$
of functions $\beta_{n}\colon I^{\left(n,\infty\right)}\to\left(0,\,\infty\right).$
Let $S_{n}\beta\colon I^{\infty}\to\left(0,\,\infty\right)$ be defined
by
\[
S_{n}\beta\left(\xi\right)=\beta_{1}\left(\xi\right)+\beta_{2}\left(\sigma\xi\right)+\cdots+\beta_{n}\left(\sigma^{n-1}\xi\right).
\]
The quantity above will determine the rate at which the shrinking
targets shrink to zero radius in the following way: Fix a sequence
$\left(\xi^{\left(n\right)}\right)$ where $\xi^{\left(n\right)}\in I^{\left(n+1,\,\infty\right)}$,
and a sequence $\left(x^{\left(n\right)}\right)\in X^{\mathbb{N}}$.
For every $\omega\in I^{n}$ we define the shrinking targets as 
\[
B_{\omega}=B\left(\varphi_{\omega}^{n}\left(x^{\left(n\right)}\right),\,e^{-S_{n}\beta\left(\omega\xi^{\left(n\right)}\right)}\right).
\]
 The shrinking target set is then defined as
\[
\mathscr{D}=\bigcap_{N=1}^{\infty}\bigcup_{n=N}^{\infty}\bigcup_{\omega\in I^{n}}B_{\omega}.
\]

As a special case one may consider the one where $\beta_{n}$ is a
constant function into $\left(0,\,\infty\right)$, as it is done in
\cite{fishman_shrinking_2015}.

We denote the Hausdorff dimension of a set $A$ by $\text{HD}\left(A\right).$
Let us also denote the diameter of a set $A$ by $\left|A\right|$.

Now for $t\ge0$ we define the upper pressure
\begin{align}
\overline{P}_{\beta}\left(t\right) & =\limsup_{n\to\infty}\frac{1}{n}\log\sum_{\omega\in I^{n}}e^{-tS_{n}\beta\left(\omega\xi^{\left(n\right)}\right)}.\label{eq:Pressure}
\end{align}

The lower pressure $\underline{P}_{\beta}\left(t\right)$ is defined
similarly by taking a limit inferior instead of a limit superior.
If $\overline{P}_{\beta}\left(t\right)=\underline{P}_{\beta}\left(t\right)$
holds, we denote this common value by $P_{\beta}\left(t\right)$. 

Now we briefly explore certain properties of the pressure functions.
Note that for $\epsilon>0$ we have that 
\begin{align*}
\overline{P}_{\beta}\left(t+\epsilon\right) & =\limsup_{n\to\infty}\frac{1}{n}\log\sum_{\omega\in I^{n}}\left(e^{-tS_{n}\beta\left(\omega\xi^{\left(n\right)}\right)}e^{-\epsilon S_{n}\beta\left(\omega\xi^{\left(n\right)}\right)}\right)\\
 & \le\limsup_{n\to\infty}\left[\frac{1}{n}\log\left(\sum_{\omega\in I^{n}}e^{-tS_{n}\beta}\right)\right]\\
 & =\overline{P}_{\beta}\left(t\right),
\end{align*}
so the upper (as well as lower) pressure function is non-increasing.
We say that the sequence $\left(\beta_{n}\right)$ is tame if the
the upper pressure is strictly decreasing. Furthermore, assuming $\#I^{\left(k\right)}\ge2$
for all $k$ it is immediate that $\underline{P}_{\beta}\left(0\right)\ge\log\left(2\right).$ 

Now, if we assume that $B>0$ such $\#I^{\left(n\right)}\le B$ for
all $n\in\mathbb{N}$, and that (\ref{eq:ESC}) holds then 
\begin{align*}
\overline{P}_{\beta}\left(d\right) & =\limsup_{n\to\infty}\frac{1}{n}\log\sum_{\omega\in I^{n}}e^{-dS_{n}\beta\left(\omega\xi^{\left(n\right)}\right)}\\
 & \le\limsup_{n\to\infty}\frac{1}{n}\log\left(B^{n}\max\left\{ e^{-dS_{n}\beta\left(\omega\xi^{\left(n\right)}\right)}\colon\omega\in I^{n}\right\} \right)\\
 & \le\limsup_{n\to\infty}\frac{1}{n}\log\left(B^{n}\overline{\kappa}_{n}^{d}e^{-nd\underline{\alpha}}\right)\\
 & \le\log B-d\underline{\alpha}+\limsup_{n\to\infty}\left(\frac{d}{n}\log\overline{\kappa}_{n}\right)\\
 & \le\log B-d\underline{\alpha}-d\theta
\end{align*}
It follows that $\overline{P}_{\beta}\left(d\right)\le0$ if $B\le e^{d\left(\theta+\underline{\alpha}\right)}.$

We observe that if $\overline{P}_{\beta}$ is strictly decreasing,
and $\overline{P}_{\beta}\left(0\right)\cdot\overline{P}_{\beta}\left(d\right)<0$,
then there exists a unique number $0<b<d$ such that

\[
b=\inf\left\{ t\ge0\mid\overline{P}_{\beta}\left(t\right)<0\right\} =\sup\left\{ t\ge0\mid\overline{P}_{\beta}\left(t\right)>0\right\} .
\]
Note that such a unique number $b$ still exists in $\left[0,\,\infty\right]$
when only assuming condition (\ref{eq:ESC}). We refer to such number
as the Bowen parameter. The main objective of our analysis is to establish
conditions under which $\text{HD}\left(\mathscr{D}\right)=b$.

\section{\label{sec:Upper-bound}Upper bound }

We say that a countable collection $\left(U_{k}\right)$ of subsets
of $\mathbb{R}^{d}$ is a $\delta$-cover of $\mathscr{D}$ if $\mathscr{D}\subset\bigcup\left(U_{k}\right)$
and $\text{diam}\left(U_{k}\right)\le\delta$ for all $k$. We recall
here the definition of $t$-dimensional Hausdorff measure. 
\begin{align*}
H^{t}\left(\mathscr{D}\right) & =\lim_{\alpha\to0}\inf\left\{ \sum_{k=1}^{\infty}\left[\text{diam}\left(U_{k}\right)\right]^{t}\mid\left(U_{k}\right)_{k\ge1}\text{ is an \ensuremath{\alpha}-cover of }\mathscr{D}\right\} \\
 & =\sup_{\alpha>0}\inf\left\{ \sum_{k=1}^{\infty}\left[\text{diam}\left(U_{k}\right)\right]^{t}\mid\left(U_{k}\right)_{k\ge1}\text{ is an \ensuremath{\alpha}-cover of }\mathscr{D}\right\} .
\end{align*}
Hausdorff dimension is then defined as 
\[
\text{HD}\left(\mathscr{D}\right)=\inf\left\{ t\ge0\mid H^{t}\left(\mathscr{D}\right)=0\right\} =\sup\left\{ t\ge0\mid H^{t}\left(\mathscr{D}\right)=\infty\right\} .
\]

\begin{thm}
\label{thm:upper bound}For any shrinking target set $\mathscr{D}$
originating from a non-autonomous IFS and a tame sequence $\beta$,
we have that $\text{HD}\left(\mathscr{D}\right)\le b$. 
\end{thm}

\begin{proof}
Let $t>b$. We will show that $H^{t}\left(\mathscr{D}\right)=0$.
Note that for any $N\ge1$ the collection $\left(\bigcup_{\omega\in I^{n}}B_{\omega}\right)_{n\ge N}$
covers $\mathscr{D}$, so 
\begin{align*}
H^{t}\left(\mathscr{D}\right) & \le\sum_{n\ge N}\sum_{\omega\in I^{n}}\left[\text{diam}\left(B_{\omega}\right)\right]^{t}\\
 & =2^{t}\sum_{n\ge N}\sum_{\omega\in I^{n}}e^{-tS_{n}\left(\omega\xi^{\left(n\right)}\right)}.
\end{align*}
 Since $t>b$ and $\beta$ is tame we have that $\overline{P}_{\beta}\left(t\right)<0$.
Thus, for large enough $M$, 
\[
n\ge M\Longrightarrow\frac{1}{n}\log\sum_{\omega\in I^{n}}e^{-tS_{n}\beta\left(\omega\xi^{\left(n\right)}\right)}<\frac{1}{2}\overline{P}_{\beta}\left(t\right)<0.
\]
Hence, 
\[
\sum_{\omega\in I^{n}}e^{-tS_{n}\beta\left(\omega\xi^{\left(n\right)}\right)}<e^{\frac{n}{2}\overline{P}_{\beta}\left(t\right)}<1.
\]
Thus, 
\[
\sum_{n\ge N}\sum_{\omega\in I^{n}}\left[\text{diam}\left(B_{\omega}\right)\right]^{t}\le2^{t}\sum_{n\ge N}e^{\frac{n}{2}\overline{P}_{\beta}\left(t\right)}.
\]
The right hand side of the inequality above is the tail of a converging
geometric series. After fixing $\epsilon>0$ we can choose $N$ large
enough so that 
\[
\sum_{n\ge N}\sum_{\omega\in I^{n}}\left[\text{diam}\left(B_{\omega}\right)\right]^{t}<\epsilon.
\]
This shows that $H^{t}\left(\mathscr{D}\right)<\epsilon$. Since $\epsilon>0$
and $t>b$ were chosen arbitrarily, we have that $\text{HD}\left(\mathscr{D}\right)\le b$.
\end{proof}

\section{\label{sec:Lower-Bound}Lower Bound}

For the proof of the lower bound we will need to impose some restrictions
on our IFS. First we establish some preliminary definitions and results. 

We define 
\begin{align*}
\underline{\kappa}_{\left(n\right)} & =\min_{j\in I^{\left(n\right)}}\inf_{x\in X}\left|D\varphi_{j}^{\left(n\right)}\left(x\right)\right|,\\
\overline{\kappa}_{\left(n\right)} & =\max_{j\in I^{\left(n\right)}}\sup_{x\in X}\left|D\varphi_{j}^{\left(n\right)}\left(x\right)\right|,\\
\underline{\kappa}_{n} & =\min_{\omega\in I^{n}}\inf_{x\in X}\left|D\varphi_{j}^{\left(n\right)}\left(x\right)\right|,\\
\overline{\kappa}_{n} & =\max_{\omega\in I^{n}}\sup_{x\in X}\left|D\varphi_{j}^{\left(n\right)}\left(x\right)\right|.
\end{align*}

It is easy to check that
\begin{equation}
\prod_{k=1}^{n}\underline{\kappa}_{\left(k\right)}\le\underline{\kappa}_{n}\le\overline{\kappa}_{n}\le\prod_{k=1}^{n}\overline{\kappa}_{\left(k\right)}.\label{eq:kappa products}
\end{equation}

Let $J$ be the limit set (attractor) of the IFS, i.e.,
\[
J=\bigcap_{n\ge1}\bigcup_{\omega\in I^{n}}\varphi_{\omega}^{n}\left(X\right).
\]

Consider the projection map $\pi_{n}\colon I^{\left(n+1,\,\infty\right)}\to X$
where $\pi_{n}\left(\xi\right)$ is defined as the element in the
singleton set 
\[
\bigcap_{k\ge1}\varphi_{\xi\vert_{k}}^{\left(n+1,\,n+k\right)}\left(X\right).
\]

We also consider a sequence of dinamically-defined sets $J_{n}$,
\[
J_{n}=\pi_{n}\left(I^{\left(n+1,\,\infty\right)}\right).
\]

We note that for every $n\in\mathbb{N}$ and every $\omega\in I^{n}$,
$\varphi_{\omega}^{n}\left(J_{n}\right)\subseteq J;$ indeed, 
\begin{align*}
\varphi_{\omega}^{n}\left(J_{n}\right) & =\varphi_{\omega}^{n}\left(\bigcup_{\xi\in I^{\left(n+1,\,\infty\right)}}\bigcap_{k\ge1}\varphi_{\xi\vert_{k}}^{\left(n+1,\,n+k\right)}\left(X\right)\right)\\
 & \subseteq\bigcup_{\xi\in I^{\left(n+1,\,\infty\right)}}\bigcap_{k\ge1}\varphi_{\omega}^{n}\left(\varphi_{\xi\vert_{k}}^{\left(n+1,\,n+k\right)}\left(X\right)\right)\\
 & =\bigcap_{k\ge1}\bigcup_{\xi\in I^{\left(n+1,\,\infty\right)}}\varphi_{\omega}^{n}\left(\varphi_{\xi\vert_{k}}^{\left(n+1,\,n+k\right)}\left(X\right)\right)\\
 & =\bigcap_{k\ge1}\bigcup_{\xi\in I^{\left(n+1,\,\infty\right)}}\varphi_{\omega\xi_{n+k}}^{n+k}\left(X\right)\\
 & =\bigcap_{k\ge1}\bigcup_{\xi\in I^{\left(n+1,\,n+k\right)}}\varphi_{\omega\xi_{n+k}}^{n+k}\left(X\right)\\
 & \subseteq\bigcap_{k\ge n+1}\bigcup_{\tau\in I^{k}}\varphi_{\tau}^{k}\left(X\right)\\
 & =J.
\end{align*}
For every $n\in\mathbb{N}\cup\left\{ 0\right\} $ we fix $\xi^{\left(n\right)}\in I^{\left(n+1,\,\infty\right)}$
and from this we define a sequence $x^{\left(n\right)}\in J_{n}$
as $x^{\left(n\right)}=\pi_{n}\left(\xi^{\left(n\right)}\right).$
This implies that the balls $B_{\omega}=B\left(\varphi_{\omega}^{n}\left(x^{\left(n\right)}\right),\,e^{-S_{n}\beta}\right)$
are centered at a point in $J$.

Furthermore, we make the following assumptions:
\begin{itemize}
\item For all $n\in\mathbb{N}$ and all $j\in I^{\left(n\right)},$ $\varphi_{j}^{\left(n\right)}$
is injective.
\item \emph{\uline{Open Set Condition (OSC)}}: For all $n\in\mathbb{N}$,
and for all $i,\,j\in I^{\left(n\right)}$, $i\ne j$, 
\[
\varphi_{\left(j\right)}^{n}\left(\text{int}\left(X\right)\right)\cap\varphi_{\left(i\right)}^{n}\left(\text{int}\left(X\right)\right)=\emptyset.
\]
\item \emph{\uline{Uniformly contracting condition (UCC)}}: Assume that
for some $\theta>0$ we have that $\overline{\kappa}_{\left(k\right)}\le e^{-\theta},$
for all $k$.
\item \emph{\uline{Exponentially shrinking condition (ESC)}}: We assume
that there exist numbers $\overline{\alpha}$ and $\underline{\alpha}$
such that
\begin{equation}
0<\underline{\alpha}\le\beta_{k}\left(\xi\right)+\log\underline{\kappa}_{\left(k\right)}\le\beta_{k}\left(\xi\right)+\log\overline{\kappa}_{\left(k\right)}\le\overline{\alpha},\label{eq:ESC}
\end{equation}
for all $k$ and all $\xi\in I^{\left(k,\infty\right)}$. It is easy
to check that 
\[
0<n\underline{\alpha}\le S_{n}\beta\left(\xi\right)+\log\underline{\kappa}_{n}\le S_{n}\beta\left(\xi\right)+\log\overline{\kappa}_{n}\le n\overline{\alpha},
\]
for all $n$ and all $\xi\in I^{\infty}$.
\item \emph{\uline{Non-empty quasi middle (NEQ):}}\emph{ }Recall that
for a set $A$ in a metric space and $\epsilon>0$, the $\epsilon$-thickening
of $A$ is 
\[
B\left(A,\,\varepsilon\right)=\bigcup_{x\in A}B\left(x,\,\varepsilon\right).
\]
Now let 
\[
X_{\varepsilon}:=X\backslash B\left(\mathbb{R}^{d}\backslash X,\,\varepsilon\right).
\]
We assume that there exists $\epsilon>0$ for which
\begin{equation}
J_{n}\cap X_{\varepsilon}\neq\emptyset,\text{ for all }n.\label{eq:NEQ}
\end{equation}
 Hence, assuming the NEQ condition we can choose the point $x^{\left(n\right)}$
appearing in the definition of the balls $B_{\omega}$ to be in $J_{n}\cap X_{\varepsilon}.$ 
\item \emph{\uline{Linear Variation Condition (LVC):}} The sequence $\left(\beta_{n}\right)$
is said have the linear variation condition if 
\[
\lim_{n\to\infty}\frac{1}{n}\left(\sup_{\xi\in I^{\infty}}S_{n}\beta\left(\xi\right)-\inf_{\overline{\xi}\in I^{\infty}}S_{n}\beta\left(\overline{\xi}\right)\right)=0.
\]
We note that this condition implies that for all $\varepsilon>0$
there exists $N_{\varepsilon}\ge1$ such that for all $n\ge N_{\varepsilon}$
and all $\xi,\,\overline{\xi}\in I^{\infty}$ we have that 
\begin{equation}
\exp\left\{ -S_{n}\beta\left(\xi\right)-\varepsilon n\right\} \le\exp\left\{ S_{n}\beta\left(\overline{\xi}\right)\right\} \le\exp\left\{ -S_{n}\beta\left(\overline{\xi}\right)+\varepsilon n\right\} .\label{eq:LVC}
\end{equation}
\item \emph{\uline{Bounded distortion property (BDP):}} We assume that
there exists $K\ge1$ such that for every $n\in\mathbb{N},$ every
$\omega\in I^{n},$ and every $x,\,y\in X$,
\[
\left|D\varphi_{\omega}^{n}\left(x\right)\right|\le K\left|D\varphi_{\omega}^{n}\left(y\right)\right|.
\]
\end{itemize}
It should be noted that a sufficient condition for BDP, one in terms
of the maps $\varphi_{a}^{\left(n\right)}$ and not in terms of the
composition $\varphi_{\omega}^{n}$, is if there exists $\alpha>0$
such that
\[
\left|\frac{\left|D\varphi_{a}^{\left(n\right)}\left(x\right)\right|}{\left|D\varphi_{a}^{\left(n\right)}\left(y\right)\right|}-1\right|\le K\left|x-y\right|^{\alpha},
\]
for all $x,\,y\in X$, all $n\in\mathbb{N},$ and all $a\in I^{\left(n\right)}.$

Let us now examine some consequences of a conformal nonautonomous
IFS having these properties. First we note that ESC and UCC imply
that the radii of $B_{\omega}$ decay exponentially fast; Indeed $e^{-S_{n}\beta\left(\omega\xi^{\left(n\right)}\right)}\le\underline{\kappa}_{n}e^{-n\theta}.$

One geometric consequence of BDP is that for every ball $B\left(x,\,r\right)\subseteq X$,
for all $n\in\mathbb{N},$ and for all $\omega\in I^{n},$ we have
that
\[
B\left(\varphi_{\omega}^{n}\left(x\right),\,K^{-1}\left\Vert D\varphi_{\omega}^{n}\right\Vert r\right)\subseteq\varphi_{\omega}^{n}\left(B\left(x,\,r\right)\right)\subseteq B\left(\varphi_{\omega}^{n}\left(x\right),\,K\left\Vert D\varphi_{\omega}^{n}\right\Vert r\right).
\]

For a proof of this fact see, for instance, \cite{mauldin_dimensions_1996}.

We remark that conformality implies the Bounded Distortion Property
whenever $d\ge2.$ For $d=2$ this follows from Koebe's distortion
theorem \cite{pommerenke_boundary_1992}, and for $d\ge3$ it is a
consequence of Liouville's theorem for conformal maps \cite{blair_inversion_2000}.

Another consequence of ESC, NEQ, and BDP is the following
\begin{claim}
\label{claim:balls in images}For all $\omega\in I^{n}$, and all
$n$ large enough, we have that $B_{\omega}\subset\varphi_{\omega}^{n}\left(X\right).$
\end{claim}

\begin{proof}
Notice that the center of the ball $B_{\omega}$ is contained in $\varphi_{\omega}^{n}\left(X_{\varepsilon}\right),$
by condition NEQ. Now,
\begin{align*}
\varphi_{\omega}^{n}\left(X\right) & \supseteq\varphi_{\omega}^{n}\left(B\left(x^{\left(n\right)},\,\varepsilon\right)\right)\\
 & \supseteq B\left(\varphi_{\omega}^{n}\left(x^{\left(n\right)}\right),\,K^{-1}\left\Vert D\varphi_{\omega}^{n}\right\Vert \varepsilon\right)\\
 & \supseteq B\left(\varphi_{\omega}^{n}\left(x^{\left(n\right)}\right),\,K^{-1}\underline{\kappa}_{n}\varepsilon\right).
\end{align*}
Thus, it suffices to show that $K^{-1}\underline{\kappa}_{n}\varepsilon\ge e^{-S_{n}\beta\left(\omega\xi^{\left(n\right)}\right)}$
for all $\omega\in I^{n}.$ Given condition ESC notice that the desired
inequality holds for all $n\ge K\left(\varepsilon\underbar{\ensuremath{\alpha}}\right)^{-1}.$
\end{proof}
Recall that a measure $\mu_{h}$ is $h$-Ahlfors regular if there
exists a constant $C\ge1$ such that 
\begin{equation}
C^{-1}\le\frac{\mu_{h}\left(B\left(x,\,r\right)\right)}{r^{h}}\le C,\label{eq:Ahlfors}
\end{equation}
for all $x\in\text{supp}\left(\mu_{h}\right)$ and all $0<r\le1.$

We establish the lower bound of the Hausdorff dimension under different
sets of assumptions. For this purpose we appeal to the celebrated
Frostmann Lemma \cite{falconer_techniques_1997}.
\begin{lem}
(Frostmann) Let $m$ be a Borel probability measure on $X$. If there
exist constants $C>0$ and $t\ge0$ such that for all $x\in X$ and
all $r>0$
\[
m\left(B\left(x,\,r\right)\right)\le Cr^{t},
\]
then $\text{HD}\left(\text{supp}\left(m\right)\right)\ge t$.
\end{lem}

\begin{thm}
\label{thm:boundedwithgaps} Let $\Phi$ be a nonautonomous conformal
IFS on a compact, convex set $X\subseteq\mathbb{R}^{d}$ with nonempty
interior satisfying OSC, ESC, UCC, LVC, and NEQ conditions. Suppose
that the sequence $\left(\frac{\overline{\kappa}_{n}}{\underline{\kappa}_{n}}\right)$
is bounded and that there exists an $h$-Ahlfors measure, $\mu_{h}$,
where $h=\text{HD}\left(J\right),$ and $\text{supp}\left(\mu_{h}\right)=J.$
Then $\text{HD}\left(\mathscr{D}\right)=b.$ 
\end{thm}

\begin{proof}
Recall that $\text{HD}\left(\mathscr{D}\right)\le b$ has been proven
in Theorem \ref{thm:upper bound}. Let $0<t<b$. Our strategy consists
of constructing a measure $m$ supported on a set $K\subseteq\mathscr{D}$
satisfying the hypothesis of the Frostmann Lemma with exponent $t$.
Choose an increasing sequence $\left(n_{l}\right)\in\mathbb{N}^{\mathbb{N}}$
such that
\begin{equation}
\overline{P}_{\beta}\left(t\right)=\lim_{l\to\infty}\frac{1}{n_{l}}\log\sum_{\omega\in I^{n}}e^{-tS_{n_{l}}\beta\left(\omega\xi^{\left(n\right)}\right)}.\label{eq:subsequence}
\end{equation}

If necessary, we refine our subsequence so that it satisfies the following
inequality for all $l$:
\begin{equation}
n_{l+1}\ge\frac{4h}{P\left(t\right)}\left(\text{const}+\overline{\alpha}\sum_{k=1}^{l}n_{k}\right).\label{eq:incsubseq1}
\end{equation}
Now define $R_{1}=I^{n_{1}}$. Assuming $R_{l}\subseteq I^{n_{l}}$
has been defined, for every $\omega\in R_{l}$ let 
\begin{align*}
R_{l+1}\left(\omega\right) & :=\left\{ \tau\in I^{n_{l+1}}\mid B_{\tau}\subset B_{\omega}\right\} .\\
R_{l+1} & :=\bigcup_{\omega\in R_{l}}R_{l+1}\left(\omega\right).
\end{align*}

Now we will focus on obtaining a lower bound on the cardinality of
the sets $R_{l+1}\left(\omega\right)$. We denote $B\left(\varphi_{\omega}^{n_{l}}\left(x^{\left(n_{l}\right)}\right),\,\frac{1}{2}e^{-S_{n_{l}}\beta\left(\omega\xi^{\left(n\right)}\right)}\right)$
by $\frac{1}{2}B_{\omega}$.
\begin{claim}
\label{claim:increasingsubsequence}Let $\tau\in I^{n_{l+1}}$ and
$\omega\in I^{n_{l}}$. If $n_{l+1}\ge\theta^{-1}\left[\log\left(2\right)+n_{l}\left(\overline{\alpha}+\theta\right)\right]$
then either 
\[
\varphi_{\tau}^{n_{l+1}}\left(X\right)\cap\frac{1}{2}B_{\omega}=\emptyset
\]
 or 
\[
\varphi_{\tau}^{n_{l+1}}\left(X\right)\subseteq B_{\omega}.
\]
\end{claim}

\begin{proof}
[Proof of Claim] Assume $\varphi_{\tau}^{n_{l+1}}\left(X\right)\cap\frac{1}{2}B_{\omega}\neq\emptyset$.
It suffices to show that $4\left|\varphi_{\tau}^{n_{l+1}}\left(X\right)\right|\le\left|B_{\omega}\right|$.
Indeed, 
\begin{align*}
4\left|\varphi_{\tau}^{n_{l+1}}\left(X\right)\right|\le\left|B_{\omega}\right| & \Longleftarrow2\overline{\kappa}_{n_{l+1}}\le e^{-S_{n_{l}}\beta\left(\omega\xi^{\left(n\right)}\right)}\\
 & \Longleftarrow S_{n_{l}}\beta\left(\omega\xi^{\left(n\right)}\right)+\log\overline{\kappa}_{n_{l+1}}\le-\log2\\
 & \Longleftarrow S_{n_{l}}\beta\left(\omega\xi^{\left(n\right)}\right)+\log\overline{\kappa}_{n_{l}}+\sum_{j=n_{l}+1}^{n_{l+1}}\log\overline{\kappa}_{\left(j\right)}\le-\log2\\
 & \Longleftarrow n_{l}\overline{\alpha}+\sum_{j=n_{l}+1}^{n_{l+1}}\log\overline{\kappa}_{\left(j\right)}\le-\log2\\
 & \Longleftarrow n_{l}\overline{\alpha}-\sum_{j=n_{l}+1}^{n_{l+1}}\theta\le-\log2\\
 & \Longleftarrow n_{l}\overline{\alpha}-\left(n_{l+1}-n_{l}\right)\theta\le-\log2\\
 & \Longleftarrow n_{l+1}\ge\theta^{-1}\left[\log\left(2\right)+n_{l}\left(\overline{\alpha}+\theta\right)\right],
\end{align*}
where the 3rd, 4rd, and 5th implications follow from (\ref{eq:kappa products}),
ESC, and UCC, respectively. This proves the Claim.
\end{proof}
From the Ahlfors property of $\mu_{h}$ we get that for all $\omega\in R_{l}$
\begin{align*}
C^{-1}\left(\frac{1}{2}e^{-S_{n_{l}}\beta\left(\omega\xi^{\left(n_{l}\right)}\right)}\right)^{h} & \le\mu_{h}\left(\frac{1}{2}B_{\omega}\right)\\
 & \le\#\left\{ \tau\in I^{n_{l+1}}\mid\varphi_{\tau}^{n_{l+1}}\left(X\right)\cap\frac{1}{2}B_{\omega}\neq\emptyset\right\} \max_{\tau\in I^{n_{l+1}}}\mu_{h}\left(\varphi_{\tau}^{n_{l+1}}\left(X\right)\right)\\
 & =\#\left\{ \tau\in I^{n_{l+1}}\mid\varphi_{\tau}^{n_{l+1}}\left(X\right)\subset B_{\omega}\right\} \max_{\tau\in I^{n_{l+1}}}\mu_{h}\left(\varphi_{\tau}^{n_{l+1}}\left(X\right)\right)\\
 & \le\#R_{l+1}\left(\omega\right)\max_{\tau\in I^{n_{l+1}}}\mu_{h}\left(\varphi_{\tau}^{n_{l+1}}\left(X\right)\right)\\
 & \le\#R_{l+1}\left(\omega\right)C\overline{\kappa}_{n_{l+1}}^{h},
\end{align*}
where the equation above follows from Claim \ref{claim:increasingsubsequence}.
Therefore, we obtain that 
\[
\#R_{l+1}\left(\omega\right)\ge C^{-2}\left(\frac{e^{-S_{n_{l}}\beta\left(\omega\xi^{\left(n_{l}\right)}\right)}}{2\overline{\kappa}_{n_{l+1}}}\right)^{h}.
\]
By redefining the constant $C$ we will write
\begin{equation}
\#R_{l+1}\left(\omega\right)\ge C^{-1}\left(\frac{e^{-S_{n_{l}}\beta\left(\omega\xi^{\left(n_{l}\right)}\right)}}{\overline{\kappa}_{n_{l+1}}}\right)^{h}.\label{eq:lowerboundR}
\end{equation}
Notice that $R_{l+1}\left(\omega\right)\neq\emptyset$ if we choose
our subsequence $\left(n_{l}\right)$ to increase rapidly enough;
indeed,
\begin{align*}
\#R_{l+1}\left(\omega\right)\ge1 & \Longleftarrow C^{-1}\left(\frac{e^{-S_{n_{l}}\beta\left(\omega\xi^{\left(n_{l}\right)}\right)}}{\overline{\kappa}_{n_{l+1}}}\right)^{h}\ge1\\
 & \Longleftarrow\overline{\kappa}_{n_{l+1}}\le C^{-\nicefrac{1}{h}}e^{-S_{n_{l}}\beta\left(\omega\xi^{\left(n_{l}\right)}\right)}\\
 & \Longleftarrow\overline{\kappa}_{n_{l}}\prod_{k=n_{l}+1}^{n_{l+1}}\overline{\kappa}_{\left(k\right)}\le C^{-\nicefrac{1}{h}}e^{-S_{n_{l}}\beta\left(\omega\xi^{\left(n_{l}\right)}\right)}\\
 & \Longleftarrow e^{-\left(n_{l+1}-n_{l}\right)\theta}\le C^{-\nicefrac{1}{h}}\overline{\kappa}_{n_{l}}^{-1}e^{-S_{n_{l}}\beta\left(\omega\xi^{\left(n_{l}\right)}\right)}\\
 & \Longleftarrow e^{-\left(n_{l+1}-n_{l}\right)\theta}\le C^{-\nicefrac{1}{h}}e^{-n_{l}\overline{\alpha}}\\
 & \Longleftarrow\left(n_{l+1}-n_{l}\right)\theta\ge\frac{1}{h}\log\left(C\right)+n_{l}\overline{\alpha}\\
 & \Longleftarrow n_{l+1}\ge\frac{1}{\theta}\left[\frac{1}{h}\log\left(C\right)+n_{l}\left(\theta+\overline{\alpha}\right)\right].
\end{align*}
Now for every $\omega\in R_{1}$ define 
\[
m_{1}\left(B_{\omega}\right)=\left(\#R_{1}\right)^{-1}.
\]
 Assuming that $m_{l}\left(B_{\omega}\right)$ has been defined for
every $\omega\in R_{l}$ we now define for every $\tau\in R_{l+1}\left(\omega\right)$
\begin{align}
m_{l+1}\left(B_{\tau}\right) & =\frac{m_{l}\left(B_{\omega}\right)}{\#R_{l+1}\left(\omega\right)}\\
 & =\left[\prod_{k=1}^{l}\left(\#R_{k+1}\left(\omega\vert_{n_{k}}\right)\right)^{-1}\right]\left(\#R_{1}\right)^{-1}.\label{eq:mass}
\end{align}
We can extend the functions $m_{l}$ to a measure on $X$ and let
us take a weak limit $m$ of the sequence $\left(m_{l}\right)$. The
function $m$ is then a Borel probability measure. Furthermore, notice
that $\text{supp}\left(m\right)\subset\text{supp}\left(m_{l}\right)=\bigcup_{\omega\in R_{l}}B_{\omega}\left(X\right)$
for all $l.$ This implies that 
\begin{align*}
K: & =\text{supp}\left(m\right)\\
 & =\bigcap_{l\ge1}\bigcup_{\omega\in R_{l}}B_{\omega}\left(X\right)\\
 & \subset\mathscr{D}.
\end{align*}
Hence, for $\tau\in R_{l+1}$ we have that $m\left(B_{\tau}\right)=m_{l+1}\left(B_{\tau}\right).$
Furthermore, from $R_{l}\neq\emptyset$ it follows that $K\ne\emptyset$.

For $\tau\in R_{l+1}\left(\omega\right)$, the inequality (\ref{eq:lowerboundR})
yields the following estimate for $m\left(B_{\tau}\right):$
\begin{align*}
m\left(B_{\tau}\right) & \le\left[\prod_{k=1}^{l}\left(\#R_{k+1}\left(\omega\vert_{n_{k}}\right)\right)^{-1}\right]\left(\#R_{1}\right)^{-1}\\
 & \le\prod_{k=1}^{l}C\left(\frac{e^{-S_{n_{k}}\beta\left(\omega\vert_{n_{k}}\xi^{\left(n_{k}\right)}\right)}}{\overline{\kappa}_{n_{k+1}}}\right)^{-h}\\
 & =C^{l}\prod_{k=1}^{l}\overline{\kappa}_{n_{k+1}}^{h}e^{hS_{n_{k}}\beta\left(\omega\vert_{n_{k}}\xi^{\left(n_{k}\right)}\right)}.
\end{align*}
Now consider $x\in K$ and a number $r$ such that $0<r<\overline{\kappa}_{n_{1}}e^{-n_{1}\overline{\alpha}}\le\min\left\{ e^{-S_{n_{1}}\beta\left(\omega\xi^{\left(n_{1}\right)}\right)}\colon\omega\in I^{n_{1}}\right\} $.
Let 
\[
\ell\left(r\right):=\min_{l\in\mathbb{N}}\left\{ l\mid\max_{\tau\in R_{l+1}}e^{-S_{n_{l+1}}\beta\left(\tau\xi^{\left(n_{l+1}\right)}\right)}\le r\right\} ,
\]
and 
\[
\#_{\ell\left(r\right)+1}:=\#\left\{ \tau\in R_{\ell\left(r\right)+1}\mid B_{\tau}\cap B\left(x,\,r\right)\neq\emptyset\right\} .
\]

Since $x\in K\subset\bigcup_{\tau\in R_{\ell\left(r\right)+1}}B_{\tau}$
it follows that $\left|x-\varphi_{\tau}^{n_{\ell\left(r\right)+1}}\left(x^{\left(n_{l}\right)}\right)\right|\le e^{-S_{n_{\ell\left(r\right)+1}}\left(\tau\xi^{\left(n_{\ell\left(r\right)+1}\right)}\right)}\le r$
for some $\tau\in R_{\ell\left(r\right)+1}$. This implies that $\varphi_{\tau}^{n_{\ell\left(r\right)+1}}\left(x^{\left(n_{l}\right)}\right)\in B\left(x,\,r\right)$
and it follows that $\#_{\ell\left(r\right)+1}\ge1$.

Recall that $m$ is supported on $K\subset\bigcup_{\tau\in R_{\ell\left(r\right)+1}}B_{\tau}$
and that $m\left(B_{\omega}\right)=m\left(B_{\overline{\omega}}\right)$
for all words $\omega,\,\overline{\omega}\in R_{l}$ of the same length,
so for all $\tau\in R_{\ell\left(r\right)+1}$ we have that

\begin{align*}
m\left(B\left(x,\,r\right)\right) & \le\#_{\ell\left(r\right)+1}\max_{\overline{\tau}\in R_{\ell\left(r\right)+1}}m\left(B_{\overline{\tau}}\right)\\
 & =\#_{\ell\left(r\right)+1}m\left(B_{\tau}\right)\\
 & \le\#_{\ell\left(r\right)+1}C^{\ell\left(r\right)}\prod_{k=1}^{\ell\left(r\right)}\overline{\kappa}_{n_{k+1}}^{h}e^{hS_{n_{k}}\beta\left(\tau\vert_{n_{k}}\xi^{\left(n_{k}\right)}\right)}\\
 & =\#_{\ell\left(r\right)+1}C^{\ell\left(r\right)}\exp\left\{ h\sum_{k=1}^{\ell\left(r\right)}S_{n_{k}}\beta\left(\tau\vert_{n_{k}}\xi^{\left(n_{k}\right)}\right)\right\} \prod_{k=1}^{\ell\left(r\right)}\overline{\kappa}_{n_{k+1}}^{h}.
\end{align*}
We will use the following upper bound for $\#_{\ell\left(r\right)+1}$.
\begin{claim}
$\#_{\ell\left(r\right)+1}\le C\left(\frac{r}{\underline{\kappa}_{n_{\ell\left(r\right)+1}}}\right)^{h}$. 
\end{claim}

\begin{proof}
[Proof of Claim] Notice that if $B_{\tau}\cap B\left(x,\,r\right)\neq\emptyset$
we have that $B_{\tau}\subset B\left(x,\,2r\right)$ since 
\[
e^{-S_{n_{\ell\left(r\right)+1}}\beta\left(\tau\xi^{\left(n_{\ell\left(r\right)+1}\right)}\right)}\le r.
\]
From the Ahlfors condition (\ref{eq:Ahlfors}) and from Claim \ref{claim:balls in images}
we get that 
\begin{align*}
Cr^{h} & \ge\mu_{h}\left(B\left(x,\,r\right)\right)\\
 & \ge\#\left\{ \tau\in I^{n_{\ell\left(r\right)+1}}\colon\varphi_{\tau}^{n_{\ell\left(r\right)+1}}\left(X\right)\cap B\left(x,\,r\right)\neq\emptyset\right\} \min_{\tau\in I^{n_{\ell\left(r\right)+1}}}\mu_{h}\left(\varphi_{\tau}^{n_{\ell\left(r\right)+1}}\left(X\right)\right)\\
 & \ge\#\left\{ \tau\in R_{\ell\left(r\right)+1}\colon\varphi_{\tau}^{n_{\ell\left(r\right)+1}}\left(X\right)\cap B\left(x,\,r\right)\neq\emptyset\right\} \min_{\tau\in I^{n_{\ell\left(r\right)+1}}}\mu_{h}\left(\varphi_{\tau}^{n_{\ell\left(r\right)+1}}\left(X\right)\right)\\
 & \ge\#\left\{ \tau\in R_{\ell\left(r\right)+1}\colon B_{\tau}\cap B\left(x,\,r\right)\neq\emptyset\right\} \min_{\tau\in I^{n_{\ell\left(r\right)+1}}}\mu_{h}\left(\varphi_{\tau}^{n_{\ell\left(r\right)+1}}\left(X\right)\right)\\
 & \ge\#_{\ell\left(r\right)+1}\min_{\tau\in I^{n_{\ell\left(r\right)+1}}}\mu_{h}\left(\varphi_{\tau}^{n_{\ell\left(r\right)+1}}\left(X\right)\right)\\
 & \ge C^{-1}\#_{\ell\left(r\right)+1}\underline{\kappa}_{n_{\ell\left(r\right)+1}}^{h}.
\end{align*}
The result follows by solving for $\#_{\ell\left(r\right)+1}.$
\end{proof}
From the previous claim we obtain that 
\[
m\left(B\left(x,\,r\right)\right)\le C^{\ell\left(r\right)}\left(\frac{r}{\underline{\kappa}_{n_{\ell\left(r\right)+1}}}\right)^{h}\exp\left\{ h\sum_{k=1}^{\ell\left(r\right)}S_{n_{k}}\beta\left(\tau\vert_{n_{k}}\xi^{\left(n_{k}\right)}\right)\right\} \prod_{k=1}^{\ell\left(r\right)}\overline{\kappa}_{n_{k+1}}^{h}.
\]
By Frostman's lemma it is enough to show that there exists $\tau\in R_{\ell\left(r\right)+1}$
for which 
\[
C^{\ell\left(r\right)}\left(\frac{r}{\underline{\kappa}_{n_{\ell\left(r\right)+1}}}\right)^{h}\exp\left\{ h\sum_{k=1}^{\ell\left(r\right)}S_{n_{k}}\beta\left(\tau\vert_{n_{k}}\xi^{\left(n_{k}\right)}\right)\right\} \prod_{k=1}^{\ell\left(r\right)}\overline{\kappa}_{n_{k+1}}^{h}\le\text{const}\cdot r^{t}
\]
holds, which is equivalent to showing that 
\[
C^{\ell}\left(\frac{\overline{\kappa}_{n_{\ell\left(r\right)+1}}}{\underline{\kappa}_{n_{\ell\left(r\right)+1}}}\right)\exp\left\{ \sum_{k=1}^{\ell\left(r\right)}S_{n_{k}}\beta\left(\tau\vert_{n_{k}}\xi^{\left(n_{k}\right)}\right)\right\} \prod_{k=1}^{\ell\left(r\right)-1}\overline{\kappa}_{n_{k+1}}\le\text{const}\cdot r^{\nicefrac{t}{h}-1}
\]
holds for some $\tau\in R_{\ell\left(r\right)+1}.$ 

From the definition of $\ell\left(r\right)$ it follows that $\exp\left\{ -S_{n_{\ell\left(r\right)}}\beta\left(\tau\vert_{n_{\ell\left(r\right)}}\xi^{\left(n_{n_{\ell\left(r\right)}}\right)}\right)\right\} >r$
for some $\tau\in R_{\ell\left(r\right)+1}.$ By comparing (\ref{eq:RU pressure})
and (\ref{eq:Pressure}) we see that $t<b\le\text{HD}\left(J\right)=h$,
so that $\frac{t}{h}<1$. Hence, we have that 
\[
\exp\left\{ \left(1-\frac{t}{h}\right)S_{n_{\ell\left(r\right)}}\beta\left(\tau\vert_{n_{\ell\left(r\right)}}\xi^{\left(n_{\ell\left(r\right)}\right)}\right)\right\} <r^{\nicefrac{t}{h}-1},
\]
for some $\tau\in R_{\ell\left(r\right)+1}$. So it suffices to show
that 
\[
C^{\ell\left(r\right)}\left(\frac{\overline{\kappa}_{n_{\ell\left(r\right)+1}}}{\underline{\kappa}_{n_{\ell\left(r\right)+1}}}\right)\exp\left\{ \sum_{k=1}^{\ell\left(r\right)}S_{n_{k}}\beta\left(\tau\vert_{n_{k}}\xi^{\left(n_{k}\right)}\right)\right\} \prod_{k=1}^{\ell\left(r\right)-1}\overline{\kappa}_{n_{k+1}}\le\text{const}\cdot\exp\left\{ \left(1-\frac{t}{h}\right)S_{n_{\ell\left(r\right)}}\beta\left(\tau\vert_{\ell\left(r\right)}\xi^{\left(\ell\left(r\right)\right)}\right)\right\} 
\]
for some $\tau\in R_{\ell\left(r\right)+1}$, which is equivalent
to showing that 
\[
C^{\ell\left(r\right)}\left(\frac{\overline{\kappa}_{n_{\ell\left(r\right)+1}}}{\underline{\kappa}_{n_{\ell\left(r\right)+1}}}\right)\exp\left\{ \sum_{k=1}^{\ell\left(r\right)-1}S_{n_{k}}\beta\left(\tau\vert_{n_{k}}\xi^{\left(n_{k}\right)}\right)\right\} \prod_{k=2}^{\ell\left(r\right)}\overline{\kappa}_{n_{k}}\le\text{const}\cdot\exp\left\{ -\frac{t}{h}S_{n_{\ell\left(r\right)}}\beta\left(\tau\vert_{n_{\ell\left(r\right)}}\xi^{\left(n_{\ell\left(r\right)}\right)}\right)\right\} 
\]
holds for some $\tau\in R_{\ell\left(r\right)+1}$.

Since $\overline{P}_{\beta}\left(t\right)>0$ we have (by choosing
$n_{1}$ large enough if necessary) that 
\[
\frac{1}{n_{\ell\left(r\right)}}\log\sum_{\omega\in I^{n_{\ell}}}\exp\left\{ -tS_{n_{\ell\left(r\right)}}\beta\left(\omega\xi^{\left(n_{\ell\left(r\right)}\right)}\right)\right\} \ge\frac{3}{4}\overline{P}_{\beta}\left(t\right),
\]
which implies that 
\[
\sum_{\omega\in I^{n_{\ell\left(r\right)}}}\exp\left\{ -tS_{n_{\ell\left(r\right)}}\beta\left(\omega\xi^{\left(n_{\ell\left(r\right)}\right)}\right)\right\} \ge\exp\left\{ \frac{n_{\ell\left(r\right)}}{2}\overline{P}_{\beta}\left(t\right)\right\} .
\]

By defining $n_{1}$ to be large enough if necessary it follows from
inequality (\ref{eq:LVC}) that for any $\tau\in R_{\ell\left(r\right)+1}$
\[
\#I^{n_{\ell\left(r\right)}}\exp\left\{ -tS_{n_{\ell}}\beta\left(\tau\vert_{n_{\ell\left(r\right)}}\xi^{\left(n_{\ell}\right)}\right)+t\varepsilon n_{\ell\left(r\right)}\right\} \ge\sum_{\omega\in I^{n_{\ell\left(r\right)}}}\exp\left\{ -tS_{n_{\ell\left(r\right)}}\beta\left(\omega\xi^{\left(n_{\ell\left(r\right)}\right)}\right)\right\} .
\]
Combining the last two inequalities we get that it suffices to show
that 
\[
\exp\left\{ -\frac{t}{h}S_{n_{\ell}}\beta\left(\tau\vert_{n_{\ell\left(r\right)}}\xi^{\left(n_{\ell}\right)}\right)\right\} \ge\left(\#I^{n_{\ell\left(r\right)}}\right)^{-\nicefrac{1}{h}}\exp\left\{ -\frac{t}{h}\varepsilon n_{\ell\left(r\right)}\right\} \exp\left\{ \frac{3n_{\ell\left(r\right)}}{4h}\overline{P}_{\beta}\left(t\right)\right\} .
\]

This estimate yields the further sufficient condition
\[
C^{\ell\left(r\right)}\left(\frac{\overline{\kappa}_{n_{\ell\left(r\right)+1}}}{\underline{\kappa}_{n_{\ell\left(r\right)+1}}}\right)\exp\left\{ \sum_{k=1}^{\ell\left(r\right)-1}S_{n_{k}}\beta\right\} \prod_{k=2}^{\ell\left(r\right)}\overline{\kappa}_{n_{k}}\le\text{const}\cdot\exp\left\{ \frac{3n_{\ell\left(r\right)}}{4h}\overline{P}_{\beta}\left(t\right)\right\} \exp\left\{ -\frac{t}{h}\varepsilon n_{\ell\left(r\right)}\right\} \left(\#I^{n_{\ell\left(r\right)}}\right)^{-\nicefrac{1}{h}}
\]
for some $\tau\in R_{\ell\left(r\right)+1}.$

If we choose $\varepsilon$ such that $0<\varepsilon<\frac{\overline{P}_{\beta}\left(t\right)}{4t}$
then it suffices to show that 
\[
C^{\ell\left(r\right)}\left(\frac{\overline{\kappa}_{n_{\ell\left(r\right)+1}}}{\underline{\kappa}_{n_{\ell\left(r\right)+1}}}\right)\exp\left\{ \sum_{k=1}^{\ell\left(r\right)-1}S_{n_{k}}\beta\right\} \prod_{k=2}^{\ell\left(r\right)}\overline{\kappa}_{n_{k}}\le\text{const}\cdot\exp\left\{ \frac{n_{\ell\left(r\right)}}{2h}\overline{P}_{\beta}\left(t\right)\right\} \left(\#I^{n_{\ell\left(r\right)}}\right)^{-\nicefrac{1}{h}}
\]
for some $\tau\in R_{\ell\left(r\right)+1}.$

Now, since $\text{supp}\left(\mu_{h}\right)=J\subseteq\bigcup_{\omega\in I^{n_{\ell\left(r\right)}}}\varphi_{\omega}\left(X\right)$
we have that
\[
1=\sum_{\omega\in I^{n_{\ell\left(r\right)}}}\mu_{h}\left(\varphi_{\omega}\left(X\right)\right)\ge C^{-1}\left(\#I^{n_{\ell\left(r\right)}}\right)\underline{\kappa}_{n_{\ell\left(r\right)}}^{h},
\]
which yields the inequality 
\[
\left(\#I^{n_{\ell\left(r\right)}}\right)^{-\nicefrac{1}{h}}\ge C^{-\nicefrac{1}{h}}\underline{\kappa}_{n_{\ell\left(r\right)}}.
\]
Hence, it is enough to show that for some $\tau\in R_{\ell\left(r\right)+1}$
\begin{equation}
C^{\ell\left(r\right)}\left(\frac{\overline{\kappa}_{n_{\ell\left(r\right)+1}}}{\underline{\kappa}_{n_{\ell\left(r\right)+1}}}\right)\left(\frac{\overline{\kappa}_{n_{\ell\left(r\right)}}}{\underline{\kappa}_{n_{\ell\left(r\right)}}}\right)\exp\left\{ \sum_{k=1}^{\ell\left(r\right)-1}S_{n_{k}}\beta\left(\tau\vert_{n_{k}}\xi^{\left(n_{k}\right)}\right)\right\} \prod_{k=2}^{\ell\left(r\right)-1}\overline{\kappa}_{n_{k}}\le\text{const}\cdot\exp\left\{ \frac{n_{\ell\left(r\right)}}{2h}\overline{P}_{\beta}\left(t\right)\right\} .\label{eq:sufficient2}
\end{equation}
 Since the sequence $\left(\frac{\overline{\kappa}_{n}}{\underline{\kappa}_{n}}\right)$
is bounded, this inequality follows by showing 
\begin{align*}
C^{\ell\left(r\right)}\exp\left\{ \sum_{k=1}^{\ell\left(r\right)-1}S_{n_{k}}\beta\left(\tau\vert_{n_{k}}\xi^{\left(n_{k}\right)}\right)\right\} \prod_{k=2}^{\ell\left(r\right)-1}\overline{\kappa}_{n_{k}} & \le\text{const}\cdot\exp\left\{ \frac{n_{\ell\left(r\right)}}{2h}\overline{P}_{\beta}\left(t\right)\right\} .\\
 & =\text{const}\cdot\exp\left\{ \frac{n_{\ell\left(r\right)}}{4h}\overline{P}_{\beta}\left(t\right)\right\} \cdot\exp\left\{ \frac{n_{\ell\left(r\right)}}{4h}\overline{P}_{\beta}\left(t\right)\right\} .
\end{align*}
 Furthermore, it is enough to show that 
\[
\exp\left\{ \sum_{k=1}^{\ell\left(r\right)-1}S_{n_{k}}\beta\left(\tau\vert_{n_{k}}\xi^{\left(n_{k}\right)}\right)\right\} \prod_{k=2}^{\ell\left(r\right)-1}\overline{\kappa}_{n_{k}}\le\text{const}\cdot\exp\left\{ \frac{n_{\ell\left(r\right)}}{4h}\overline{P}_{\beta}\left(t\right)\right\} 
\]
for some $\tau\in R_{\ell\left(r\right)+1}$ and that 
\[
C^{\ell\left(r\right)}\le\exp\left\{ \frac{n_{\ell\left(r\right)}}{4h}\overline{P}_{\beta}\left(t\right)\right\} .
\]
The first inequality is satisfied given condition (\ref{eq:incsubseq1}).
The second inequality is satisfied by choosing our rapidly increasing
sequence $\left(n_{l}\right)$ to satisfy $n_{l}\gg l.$ This completes
the proof.
\end{proof}
\begin{thm}
\label{thm:unboundedwithgaps}Let $\Phi$ be a conformal nonautonomous
IFS satisfying the OSC, ESC, UCC, and NEQ conditions. If there exists
an $h$-Ahlfors measure supported on $J$, $\overline{P}\left(t\right)=\underline{P}\left(t\right)$
on a neighborhood of $b$, and 
\begin{equation}
\lim_{n\to\infty}\frac{1}{n}\log\frac{\overline{\kappa}_{n}}{\underline{\kappa}_{n}}=0,\label{eq:subexponential}
\end{equation}
then $\text{HD}\left(\mathscr{D}\right)=b$.
\end{thm}

\begin{proof}
As before, we choose $0\le t<b$ in the neighborhood of $b$ where
$P_{\beta}$ exists. It suffices to show that inequality (\ref{eq:sufficient2})
holds. This will follow from showing that the following three inequalities
hold for some $\tau\in R_{\ell\left(r\right)+1}$:

\begin{align}
C^{\ell\left(r\right)}\exp\left\{ \sum_{k=1}^{\ell\left(r\right)-1}S_{n_{k}}\beta\left(\tau\vert_{n_{k}}\xi^{\left(n_{k}\right)}\right)\right\} \prod_{k=2}^{\ell\left(r\right)-1}\overline{\kappa}_{n_{k}} & \le\text{const}\cdot\exp\left\{ \frac{n_{\ell\left(r\right)}}{6h}P_{\beta}\left(t\right)\right\} ,\label{eq:threeinequalities}
\end{align}
~

\[
\frac{\overline{\kappa}_{n_{\ell\left(r\right)}}}{\underline{\kappa}_{n_{\ell\left(r\right)}}}\le\exp\left\{ \frac{n_{\ell\left(r\right)}}{6h}P_{\beta}\left(t\right)\right\} ,
\]
and 
\[
\frac{\overline{\kappa}_{n_{\ell\left(r\right)+1}}}{\underline{\kappa}_{n_{\ell\left(r\right)+1}}}\le\exp\left\{ \frac{n_{\ell\left(r\right)}}{6h}P_{\beta}\left(t\right)\right\} .
\]
The second inequality is equivalent to the inequality
\[
\frac{1}{n_{\ell\left(r\right)}}\log\frac{\overline{\kappa}_{n_{\ell\left(r\right)}}}{\underline{\kappa}_{n_{\ell\left(r\right)}}}\le\frac{P_{\beta}\left(t\right)}{6h},
\]
which is satisfied simply by choosing $n_{1}$ large enough. This
can be achieved withough loss of generality since $P_{\beta}\left(t\right)>0$
and by assumption (\ref{eq:subexponential}).

To proving the third inequality first we note that it is equivalent
to 
\[
\frac{n_{\ell\left(r\right)+1}}{n_{\ell\left(r\right)}}\frac{1}{n_{\ell\left(r\right)+1}}\log\frac{\overline{\kappa}_{n_{\ell\left(r\right)}}}{\underline{\kappa}_{n_{\ell\left(r\right)}}}\le\frac{P_{\beta}\left(t\right)}{6h}.
\]
Let 
\[
0<A\le\frac{P_{\beta}\left(t\right)}{6h}\underline{\alpha}\left(\log C+\overline{\alpha}\right)^{-1},
\]
where $C$ is the same constant as in (\ref{eq:Ahlfors}). Since $\overline{P}_{\beta}\left(t\right)=\underline{P}_{\beta}\left(t\right)$
we have that (\ref{eq:subsequence}) holds for every increasing sequence
$\left(n_{l}\right)$. Consider in particular an increasing sequence
with the property 
\begin{equation}
n_{l+1}=\min\left\{ n\in\mathbb{N}\mid nA\ge\underline{\alpha}\left(n_{1}+\cdots+n_{l}\right)\right\} ,\label{eq:subseq}
\end{equation}
for all $l\in\mathbb{N}.$

Such a sequence satisfies the following claim.
\begin{claim}
\label{claim:boundedconsecutiveterms}The following inequality holds:
\[
A^{-1}\underline{\alpha}\le\frac{n_{l+1}}{n_{l}}\le A^{-1}\underline{\alpha}+2.
\]
\end{claim}

\begin{proof}
Condition \ref{eq:subseq} implies that
\[
A\left(n_{l+1}-1\right)\le\underline{\alpha}\left(n_{1}+\cdots+n_{l}\right)\le An_{l+1}.
\]
Therefore, 
\begin{align*}
A\left(n_{l+1}-n_{l}-1\right) & \le\underline{\alpha}n_{l}\le A\left(n_{l+1}-n_{l}+1\right)
\end{align*}
Re-arranging terms algebraically we get 
\[
\underline{\alpha}n_{l}A^{-1}-1\le n_{l+1}-n_{l}\le\underline{\alpha}n_{l}A^{-1}+1,
\]
\[
n_{l}+\underline{\alpha}n_{l}A^{-1}-1\le n_{l+1}\le n_{l}+\underline{\alpha}n_{l}A^{-1}+1,
\]
\begin{align*}
\underline{\alpha}A^{-1}+1-\frac{1}{n_{l}} & \le\frac{n_{l+1}}{n_{l}}\le\underline{\alpha}A^{-1}+1+\frac{1}{n_{l}},
\end{align*}
\[
\underline{\alpha}A^{-1}\le\frac{n_{l+1}}{n_{l}}\le\underline{\alpha}A^{-1}+2.
\]
This proves the claim.
\end{proof}
Since our sequence is chosen so that $\frac{n_{l+1}}{n_{l}}$ is uniformly
bounded, the desired inequality 
\[
\frac{n_{\ell\left(r\right)+1}}{n_{\ell\left(r\right)}}\frac{1}{n_{\ell\left(r\right)+1}}\log\frac{\overline{\kappa}_{n_{\ell\left(r\right)+1}}}{\underline{\kappa}_{n_{\ell\left(r\right)+1}}}\le\frac{P_{\beta}\left(t\right)}{6h}
\]
follows again by choosing $n_{1}$ large enough. 

The remaining inequality 
\[
C^{\ell\left(r\right)}\exp\left\{ \sum_{k=1}^{\ell\left(r\right)-1}S_{n_{k}}\beta\left(\tau\vert_{n_{k}}\xi^{\left(n_{k}\right)}\right)\right\} \prod_{k=2}^{\ell\left(r\right)-1}\overline{\kappa}_{n_{k}}\le\text{const}\cdot\exp\left\{ \frac{n_{\ell\left(r\right)}}{6h}P_{\beta}\left(t\right)\right\} 
\]
is equivalent to showing
\[
\ell\left(r\right)\log\left(C\right)+\sum_{k=1}^{\ell\left(r\right)-1}\left(S_{n_{k}}\beta\left(\tau\vert_{n_{k}}\xi^{\left(n_{k}\right)}\right)+\log\overline{\kappa}_{n_{k}}\right)\le\text{const}+\frac{n_{\ell\left(r\right)}}{6h}P_{\beta}\left(t\right).
\]
Given ESC, it suffices to show
\[
\ell\left(r\right)\log\left(C\right)+\sum_{k=1}^{\ell\left(r\right)-1}n_{k}\overline{\alpha}\le\text{const}+\frac{n_{\ell\left(r\right)}}{6h}P_{\beta}\left(t\right).
\]

Since $\text{const}>0$, this inequality is follows from showing
\begin{equation}
n_{\ell\left(r\right)}\ge\frac{6h}{P_{\beta}\left(t\right)}\left[\ell\left(r\right)\log\left(C\right)+\sum_{k=1}^{\ell\left(r\right)-1}n_{k}\overline{\alpha}\right].\label{eq:incsubseq3}
\end{equation}

In view of Claim \ref{claim:boundedconsecutiveterms}, we have that
$n_{l}\ge A^{-1}\underbar{\ensuremath{\alpha}}\left(n_{1}+\cdots+n_{l-1}\right)$
for all $l$. Now, condition (\ref{eq:incsubseq3}) holds if 
\[
A^{-1}\underbar{\ensuremath{\alpha}}\sum_{k=1}^{l-1}n_{k}\ge\frac{6h}{P_{\beta}\left(t\right)}\left[l\cdot\log\left(C\right)+\overline{\alpha}\sum_{k=1}^{l-1}n_{k},\right]
\]
or, re-arranging terms, if
\[
A^{-1}\ge\frac{6h}{P_{\beta}\left(t\right)}\left[\frac{l}{n_{1}+\cdots+n_{l-1}}\cdot\frac{\log\left(C\right)}{\underline{\alpha}}+\frac{\overline{\alpha}}{\underline{\alpha}}\right].
\]
Since the sequence $\left(n_{l}\right)$ is increasing and assuming
without loss of generality that $n_{1}\ge2$, we have that $l\le n_{1}+\cdots+n_{l-1}$for
all $l$. Hence, it suffices to show that 
\[
A^{-1}\ge\frac{6h}{P_{\beta}\left(t\right)}\left[\frac{\log\left(C\right)+\overline{\alpha}}{\underline{\alpha}}\right].
\]
This follows from our choice of $A$ above.

Since all three inequalities in (\ref{eq:threeinequalities}) hold,
this completes the proof.
\end{proof}

\section{\label{sec:Ahlfors-Measures}Ahlfors Measures}

Now we focus our attention on stablishing sufficient conditions for
the existance of an $h$-Ahlfors measure. Let us define for every
$n\in\mathbb{N},$ 
\[
\rho_{n}=\max_{a,\,b\in I^{\left(n\right)}}\frac{\left\Vert D\varphi_{a}^{\left(n\right)}\right\Vert }{\left\Vert D\varphi_{b}^{\left(n\right)}\right\Vert },
\]
and 
\[
Z_{n}\left(t\right)=\sum_{\omega\in I^{n}}\left\Vert D\varphi_{\omega}^{n}\right\Vert ^{t}.
\]
Following the analysis in \cite{rempe-gillen_non-autonomous_2015}
we obtain the following result.
\begin{thm}
If the sequences $\left(\#I^{\left(n\right)}\right)_{n\ge1},\,\left(\rho_{n}\right)_{n\ge1},\,\left(Z_{n}\left(h\right)\right)_{n\ge1},\text{ and }\left(Z_{n}^{-1}\left(h\right)\right)_{n\ge1}$
are bounded, then there exists an $h$-Ahlfors measure supported on
$J$. 
\end{thm}

\begin{proof}
In the proof of Therem 3.2 in \cite{rempe-gillen_non-autonomous_2015}
the authors construct a measure $\mu$ on $J$ for which 
\[
\mu\left(B\left(x,\,r\right)\right)\le Cr^{t}
\]
holds for every $x\in X$ and $r>0$ and for every $t\ge0$ satisfying
\[
\liminf_{n\to\infty}\frac{Z_{n-1}\left(t\right)\cdot\left(\#I^{\left(n\right)}\right)^{\nicefrac{t}{d}}}{1+\log\left[\max_{j\le n}\rho_{j}\right]}\min_{a\in I^{\left(n\right)}}\left\Vert D\varphi_{a}^{\left(n\right)}\right\Vert ^{t}>0.
\]
We claim that the measure $\mu$ is $h$-Ahlfors. In order to prove
the upper bound in the Ahlfors condition it suffices to show that
the limit inferior above is positive for $t=h$. 

Let $B$ be a bound for all the sequences in the hypothesis of the
theorem. Since $\#I^{\left(n\right)}\ge2$ and $\rho_{n}\le B$, it
suffices to show that 
\[
\liminf_{n\to\infty}Z_{n-1}\left(h\right)\min_{a\in I^{\left(n\right)}}\left\Vert D\varphi_{a}^{\left(n\right)}\right\Vert ^{h}>0.
\]

Note that since the sequence $\left(Z_{n}^{-1}\left(h\right)\right)_{n\ge1}$
is bounded abounded above by $B$ we have that the sequence $\left(Z_{n}\left(h\right)\right)_{n\ge1}$
is bounded below by $B^{-1}>0$. 

So it suffices to show the following
\begin{claim}
\label{claim:BoundedNorms}The sequence 
\[
\left(\min_{a\in I^{\left(n\right)}}\left\Vert D\varphi_{a}^{\left(n\right)}\right\Vert ^{h}\right)_{n\ge1}
\]
is bounded below by a positive number.
\end{claim}

\begin{proof}
Note that 
\begin{align*}
B^{-1} & \le Z_{n+1}\left(h\right)\\
 & =\sum_{\tau\in I^{n+1}}\left\Vert D\varphi_{\tau}^{n+1}\right\Vert ^{h}\\
 & =\sum_{\omega\in I^{n}}\sum_{a\in I^{\left(n+1\right)}}\left\Vert D\varphi_{\omega a}^{n+1}\right\Vert ^{h}\\
 & \le\sum_{\omega\in I^{n}}\sum_{a\in I^{\left(n+1\right)}}\left\Vert D\varphi_{\omega}^{n}\right\Vert ^{h}\left\Vert D\varphi_{a}^{\left(n+1\right)}\right\Vert ^{h}\\
 & =Z_{n}\left(h\right)\sum_{a\in I^{\left(n+1\right)}}\left\Vert D\varphi_{a}^{\left(n+1\right)}\right\Vert ^{h}\\
 & \le Z_{n}\left(h\right)\left(\#I^{\left(n+1\right)}\right)\max_{a\in I^{\left(n+1\right)}}\left\Vert D\varphi_{a}^{\left(n+1\right)}\right\Vert ^{h}\\
 & =Z_{n}\left(h\right)\left(\#I^{\left(n+1\right)}\right)\rho_{n+1}\min_{a\in I^{\left(n+1\right)}}\left\Vert D\varphi_{a}^{\left(n+1\right)}\right\Vert ^{h}
\end{align*}
Since the product $Z_{n}\left(h\right)\left(\#I^{\left(n+1\right)}\right)\rho_{n+1}$
is uniformly bounded above, the claim follows.
\end{proof}
To prove that $\mu\left(B\left(x,\,r\right)\right)\ge C^{-1}r^{t}$
we shall now consider an arbitrary $0\le r<\text{diam}\left(J\right)$
and $x\in J.$ Note that 
\[
x\in\bigcap_{n\ge1}\varphi_{\xi\vert_{n}}^{n}\left(X\right),
\]
for some $\xi\in I^{\infty}.$ Define 
\[
\mathfrak{n}=\max\left\{ k\in\mathbb{N}\mid\text{diam}\left(\varphi_{\xi\vert_{k}}^{k}\left(X\right)\right)\ge r\right\} .
\]
It follows that 
\begin{align*}
\mu\left(B\left(x,\,r\right)\right) & \ge\mu\left(\varphi_{\xi\vert_{\mathfrak{n}+1}}^{\mathfrak{n}+1}\left(X\right)\right)
\end{align*}
since $\varphi_{\xi\vert_{\mathfrak{n}+1}}^{\mathfrak{n}+1}\left(X\right)\subseteq B\left(x,\,r\right)$. 

We make the following 
\begin{claim}
For all $n\in\mathbb{N}$ and every $\omega\in I^{n}$ the measure
$\mu$ satisfies
\[
\mu\left(\varphi_{\omega}^{n}\left(X\right)\right)\ge K^{-h}Z_{n}^{-1}\left(h\right)\left\Vert D\varphi_{\omega}^{n}\right\Vert ^{h},
\]
where $K\ge1$ is the distortion constant.
\end{claim}

\begin{proof}
In \cite{rempe-gillen_non-autonomous_2015} the measure $\mu$ is
constructed as a weak limit of a sequence $\left(\mu^{\left(n\right)}\right)_{n\ge1}$
of measures where 
\[
\mu^{\left(n\right)}\left(\varphi_{\omega}^{n}\left(X\right)\right)=\frac{\left\Vert D\varphi_{\omega}^{n}\right\Vert ^{h}}{Z_{n}\left(h\right)}
\]
 for all $\omega\in I^{n}.$ Now, for every $q\in\mathbb{N}$ and
every $\omega\in I^{n}$ we have that 
\begin{align*}
\mu^{\left(n+q\right)}\left(\varphi_{\omega}^{n}\left(X\right)\right) & =\mu^{\left(n+q\right)}\left(\varphi_{\omega}^{n}\left(X\right)\cap\text{supp}\left(\mu^{\left(n+q\right)}\right)\right)\\
 & =\mu^{\left(n+q\right)}\left(\bigcup_{\gamma\in I^{\left(n+1,n+q\right)}}\varphi_{\omega\gamma}^{n+q}\left(X\right)\right)\\
 & =\sum_{\gamma\in I^{\left(n+1,n+q\right)}}\mu^{\left(n+q\right)}\left(\varphi_{\omega\gamma}^{n+q}\left(X\right)\right)\\
 & =Z_{n+q}^{-1}\left(h\right)\sum_{\gamma\in I^{\left(n+1,n+q\right)}}\left\Vert D\varphi_{\omega\gamma}^{n+q}\right\Vert ^{h}\\
 & \ge Z_{n+q}^{-1}\left(h\right)\sum_{\gamma\in I^{\left(n+1,n+q\right)}}K^{-h}\left\Vert D\varphi_{\omega}^{n}\right\Vert ^{h}\left\Vert D\varphi_{\gamma}^{\left(n+1,n+q\right)}\right\Vert ^{h},
\end{align*}
where the last inequality follows from the BDP. 

Furthermore, the inequality 
\[
Z_{n+q}^{-1}\left(h\right)\sum_{\gamma\in I^{\left(n+1,n+q\right)}}\left\Vert D\varphi_{\gamma}^{\left(n+1,n+q\right)}\right\Vert ^{h}\ge Z_{n}^{-1}\left(h\right)
\]
follows from noting that
\begin{align*}
Z_{n+q}\left(h\right) & =\sum_{\tau\in I^{n+q}}\left\Vert D\varphi_{\tau}^{n+q}\right\Vert ^{h}\\
 & =\sum_{\omega'\in I^{n}}\sum_{\gamma\in I^{\left(n+1,n+q\right)}}\left\Vert D\varphi_{\omega'\gamma}^{n+q}\right\Vert ^{h}\\
 & \le\sum_{\omega'\in I^{n}}\sum_{\gamma\in I^{\left(n+1,n+q\right)}}\left\Vert D\varphi_{\omega'}^{n}\right\Vert ^{h}\left\Vert D\varphi_{\gamma}^{\left(n+1,n+q\right)}\right\Vert ^{h}\\
 & =Z_{n}\left(h\right)\sum_{\gamma\in I^{\left(n+1,n+q\right)}}\left\Vert D\varphi_{\gamma}^{\left(n+1,n+q\right)}\right\Vert ^{h}.
\end{align*}
This proves that 
\[
\mu_{n+q}\left(\varphi_{\omega}^{n}\left(X\right)\right)\ge Z_{n}^{-1}\left(h\right)K^{-h}\left\Vert D\varphi_{\omega}^{n}\right\Vert ^{h}
\]
for all $q\in\mathbb{N}$. Taking the limit as $q\to\infty$ proves
the claim.
\end{proof}
From the claim above it follows now that 
\begin{align*}
\mu\left(B\left(x,\,r\right)\right) & \ge\mu\left(\varphi_{\xi\vert_{\mathfrak{n}+1}}^{\mathfrak{n}+1}\left(X\right)\right)\\
 & \ge C^{-1}Z_{\mathfrak{n}+1}^{-1}\left(h\right)\left\Vert D\varphi_{\xi\vert_{\mathfrak{n}+1}}^{\mathfrak{n}+1}\right\Vert ^{h}\\
 & \ge C^{-1}Z_{\mathfrak{n}+1}^{-1}\left(h\right)K^{-h}\left\Vert D\varphi_{\xi\vert_{\mathfrak{n}}}^{\mathfrak{n}}\right\Vert ^{h}\left\Vert D\varphi_{\xi_{\mathfrak{n}+1}}^{\left(\mathfrak{n}+1\right)}\right\Vert ^{h},
\end{align*}
where the last inequality follows from BDP. 

By the mean value inequality we have that 
\[
\left\Vert D\varphi_{\xi\vert_{\mathfrak{n}}}^{\mathfrak{n}}\right\Vert \text{diam}\left(X\right)\ge\text{diam}\left(\varphi_{\xi\vert_{\mathfrak{n}}}^{\mathfrak{n}}\left(X\right)\right)\ge r.
\]
Redefining $C$ we obtain that
\[
\mu\left(B\left(x,\,r\right)\right)\ge C^{-1}Z_{\mathfrak{n}+1}^{-1}\left(h\right)r^{h}\left\Vert D\varphi_{\xi_{\mathfrak{n}+1}}^{\left(\mathfrak{n}+1\right)}\right\Vert ^{h}.
\]
From the hypothesis and Claim \ref{claim:BoundedNorms} the product
$Z_{n+1}^{-1}\left(h\right)\left\Vert D\varphi_{\xi_{n+1}}^{\left(n+1\right)}\right\Vert ^{h}$
is uniformly bounded below by a positive number. This allows us to
redefine $C$, independent of $x$ and $r$, to obtain 
\[
\mu\left(B\left(x,\,r\right)\right)\ge C^{-1}r^{h},
\]
as desired.
\end{proof}

\section{Perturbations of linear systems in one dimension}

Let $X=\left[0,\,1\right]$ and consider a piecewise linear nonautonomous
IFS $\Phi=\left\{ \varphi_{e}^{\left(n\right)}\right\} _{n\in\mathbb{N},\,e\in I^{\left(n\right)}}$.
Now consider a nonlinear perturbative system $\tilde{\Phi}=\left\{ \tilde{\varphi}_{e}^{\left(n\right)}\right\} _{n\in\mathbb{N},\,e\in I^{\left(n\right)}}$
satisfying 
\[
\tilde{\varphi}_{e}^{\left(n\right)}\left(x\right)=\varphi_{e}^{\left(n\right)}\left(u_{e}^{\left(n\right)}\right)+\left(\left(\varphi_{e}^{\left(n\right)}\right)'\right)\int_{u_{e}^{\left(n\right)}}^{x}\left(1+\gamma_{e}^{\left(n\right)}\left(t\right)\right)\,dt,
\]
where $u_{e}^{\left(n\right)}\in\left[0,\,1\right]$ and $\gamma_{e}^{\left(n\right)}\colon\left[0,\,1\right]\to\left(-\epsilon_{n},\,\epsilon_{n}\right)$
is H\"older continuous and $\epsilon_{n}>0$ is independent of $e\in I^{\left(n\right)}$.
Our goal is to establish sufficient conditions on $\gamma'\text{s}$
for which  the system $\left\{ \tilde{\varphi}\right\} $ satisfies
the hypothesis of Theorem \ref{thm:boundedwithgaps} or \ref{thm:unboundedwithgaps}.

Observe that 
\[
\left|\left(\tilde{\varphi}_{e}^{\left(n\right)}\right)'\left(x\right)\right|=\left|\left(\varphi_{e}^{\left(n\right)}\right)'\right|\left|1+\gamma_{e}^{\left(n\right)}\left(x\right)\right|,
\]
and
\begin{align*}
\left|\left(\tilde{\varphi}_{\omega}^{n}\right)'\left(x\right)\right| & =\prod_{k=1}^{n}\left(\tilde{\varphi}_{\omega_{k}}^{\left(k\right)}\right)'\left(\tilde{\varphi}_{\sigma^{k}\omega}^{\left(k+1,\,n\right)}\left(x\right)\right)\\
 & =\left|\left(\varphi_{\omega}^{n}\right)'\right|\prod_{k=1}^{n}\left|1+\gamma_{\omega_{k}}^{\left(k\right)}\left(\tilde{\varphi}_{\sigma^{k}\omega}^{\left(k+1,\,n\right)}\left(x\right)\right)\right|\\
 & \le\overline{\kappa}_{n}\prod_{k=1}^{n}\left|1+\gamma_{\omega_{k}}^{\left(k\right)}\left(\tilde{\varphi}_{\sigma^{k}\omega}^{\left(k+1,\,n\right)}\left(x\right)\right)\right|.
\end{align*}
Now define 
\[
\overline{\gamma^{\left(n\right)}}=\max_{e\in I^{\left(n\right)}}\sup_{x\in\left[0,\,1\right]}\left\{ \left|\gamma_{e}^{\left(n\right)}\left(x\right)\right|\right\} .
\]
Then we have that for all $x\in\left[0,\,1\right]$
\[
\underline{\kappa}_{n}\prod_{k=1}^{n}\left[1-\overline{\gamma^{\left(k\right)}}\right]\le\left|\left(\tilde{\varphi}_{\omega}^{n}\right)'\left(x\right)\right|\le\overline{\kappa}_{n}\prod_{k=1}^{n}\left[1+\overline{\gamma^{\left(k\right)}}\right].
\]

Now we impose some conditions on $\epsilon$ that will guarantee $\left\{ \tilde{\varphi}\right\} $
to satisfy the OSC. Let $g_{\left(n\right)}$ be the size of the smallest
``gap'' between images under the unperturbed system $\Phi$ at level
$n,$ i.e., 
\[
g_{\left(n\right)}:=\min\left\{ \left|\varphi_{j}^{\left(n\right)}\left(x\right)-\varphi_{i}^{\left(n\right)}\left(y\right)\right|\colon x,\,y\in\left[0,\,1\right];\,j,\,i\in I^{\left(n\right)},\,j\neq i\right\} .
\]

We will assume $\Phi$ has the strong separation condition, i.e.,
that $g_{\left(n\right)}>0$ for all $n$.
\begin{lem}
If $0<\epsilon_{n}<\frac{g_{\left(n\right)}}{2\overline{\kappa}_{\left(n\right)}}$
for all $n$, then $\tilde{\Phi}$ has the strong separation condition. 
\end{lem}

\begin{proof}
Observe that 
\begin{align*}
\left|\tilde{\varphi}_{j}^{\left(n\right)}\left(x\right)-\varphi_{j}^{\left(n\right)}\left(x\right)\right| & =\left|\varphi_{j}^{\left(n\right)}\left(u_{j}^{\left(n\right)}\right)+\left(\left(\varphi_{j}^{\left(n\right)}\right)'\right)\int_{u_{j}^{\left(n\right)}}^{x}\left(1+\gamma_{j}^{\left(n\right)}\left(t\right)\right)\,dt-\varphi_{j}^{\left(n\right)}\left(u_{j}^{\left(n\right)}\right)-\int_{u_{j}^{\left(n\right)}}^{x}\left(\varphi_{j}^{\left(n\right)}\right)'\left(t\right)\,dt\right|\\
 & \le\left|\left(\left(\varphi_{j}^{\left(n\right)}\right)'\right)\left(1+\epsilon_{n}\right)\left(x-u_{j}^{\left(n\right)}\right)-\left(\varphi_{j}^{\left(n\right)}\right)'\left(x-u_{j}^{\left(n\right)}\right)\right|\\
 & \le\epsilon_{n}\left|\left(\varphi_{j}^{\left(n\right)}\right)'\right|\\
 & \le\epsilon_{n}\overline{\kappa}_{\left(n\right)}\\
 & <\frac{g_{\left(n\right)}}{2}.
\end{align*}
Note that the right hand side is independent of $j\in I^{\left(n\right)}$.
Now, it is an elementary fact in analisys that $\left|a+b+c\right|\ge\left|a\right|-\left|b\right|-\left|c\right|$
for all $a,\,b,\,c\in\mathbb{R}$. Using this inequality we show that
\begin{align*}
\left|\tilde{\varphi}_{j}^{\left(n\right)}\left(x\right)-\tilde{\varphi}_{i}^{\left(n\right)}\left(y\right)\right| & =\left|\tilde{\varphi}_{j}^{\left(n\right)}\left(x\right)-\varphi_{j}^{\left(n\right)}\left(x\right)+\varphi_{j}^{\left(n\right)}\left(x\right)-\varphi_{i}^{\left(n\right)}\left(y\right)+\varphi_{i}^{\left(n\right)}\left(y\right)-\tilde{\varphi}_{i}^{\left(n\right)}\left(y\right)\right|\\
 & \ge\left|\varphi_{j}^{\left(n\right)}\left(x\right)-\varphi_{i}^{\left(n\right)}\left(y\right)\right|-\left|\tilde{\varphi}_{j}^{\left(n\right)}\left(x\right)-\varphi_{j}^{\left(n\right)}\left(x\right)\right|-\left|\varphi_{i}^{\left(n\right)}\left(y\right)-\tilde{\varphi}_{i}^{\left(n\right)}\left(y\right)\right|\\
 & >g_{\left(n\right)}-\frac{g_{\left(n\right)}}{2}-\frac{g_{\left(n\right)}}{2}\\
 & =0.
\end{align*}
 for all $j\ne i$ and all $x,\,y\in\left[0,\,1\right].$
\end{proof}
Furthermore, define
\begin{align*}
\overline{\tilde{\kappa}}_{\left(n\right)} & =\max_{e\in I^{\left(n\right)}}\sup_{x\in X}\left|\left(\tilde{\varphi}_{e}^{\left(n\right)}\right)'\left(x\right)\right|\\
\underline{\tilde{\kappa}}_{\left(n\right)} & =\min_{e\in I^{\left(n\right)}}\inf_{x\in X}\left|\left(\tilde{\varphi}_{e}^{\left(n\right)}\right)'\left(x\right)\right|\\
\overline{\tilde{\kappa}}_{n} & =\max_{\omega\in I^{n}}\sup_{x\in X}\left|\left(\tilde{\varphi}_{\omega}^{n}\right)'\left(x\right)\right|\\
\underline{\tilde{\kappa}}_{n} & =\min_{\omega\in I^{n}}\inf_{x\in X}\left|\left(\tilde{\varphi}_{\omega}^{n}\right)'\left(x\right)\right|.
\end{align*}
Observe that 
\[
\prod_{j=1}^{n}\underline{\tilde{\kappa}}_{\left(j\right)}\le\underline{\tilde{\kappa}}_{n}\le\overline{\tilde{\kappa}}_{n}\le\prod_{j=1}^{n}\overline{\tilde{\kappa}}_{\left(j\right)}.
\]

Now, 
\[
\frac{\overline{\tilde{\kappa}}_{n}}{\underline{\tilde{\kappa}}_{n}}\le\frac{\overline{\kappa}_{n}}{\underline{\kappa}_{n}}\prod_{k=1}^{n}\frac{1+\overline{\gamma^{\left(k\right)}}}{1-\overline{\gamma^{\left(k\right)}}}.
\]

Since $\left|\gamma\right|\le1-\epsilon$ it follows that $\underline{\tilde{\kappa}}_{\left(n\right)}>0$
for all $n$. From this estimate we see that the sequence $\left(\frac{\overline{\tilde{\kappa}}_{n}}{\underline{\tilde{\kappa}}_{n}}\right)$
is bounded if $\left(\frac{\overline{\kappa}_{n}}{\underline{\kappa}_{n}}\right)$
is bounded and if $\sup_{n\ge1}\prod_{k=1}^{n}\frac{1+\overline{\gamma^{\left(k\right)}}}{1-\overline{\gamma^{\left(k\right)}}}<\infty$.
Note that 
\begin{align*}
\sup_{n\ge1}\prod_{k=1}^{n}\frac{1+\overline{\gamma^{\left(k\right)}}}{1-\overline{\gamma^{\left(k\right)}}}<\infty & \Longleftrightarrow\sum_{k\ge1}\log\left(\frac{1+\overline{\gamma^{\left(k\right)}}}{1-\overline{\gamma^{\left(k\right)}}}\right)<\infty\\
 & \Longleftrightarrow\sum_{k\ge1}\overline{\gamma^{\left(k\right)}}<\infty,
\end{align*}
where the last step follows from the limit comparison test in calculus.
Hence, we have the following
\begin{thm}
Suppose that $\Phi$ is a nonautonomous IFS of linear functions satisfying
the hypothesis of Theorem \ref{thm:boundedwithgaps} and that $\left(\gamma_{e}^{\left(n\right)}\left(x\right)\right)_{n\in\mathbb{N},\,e\in I^{\left(n\right)}}$
is a sequence of H\"older-continuous functions from $\left[0,\,1\right]$
into $\left(-\epsilon_{n},\,\epsilon_{n}\right)$ for some $\epsilon_{n}\in\left(0,\,1\right).$
Furthermore, let $\tilde{\Phi}$ be a nonlinear perturbation of $\Phi$
defined as
\[
\tilde{\varphi}_{e}^{\left(n\right)}\left(x\right)=\varphi_{e}^{\left(n\right)}\left(u_{e}^{\left(n\right)}\right)+\left(\left(\varphi_{e}^{\left(n\right)}\right)'\right)\int_{u_{e}^{\left(n\right)}}^{x}\left(1+\gamma_{e}^{\left(n\right)}\left(t\right)\right)\,dt.
\]
 If 
\[
\sum_{k\ge1}\overline{\gamma^{\left(k\right)}}<\infty,
\]
and either
\begin{enumerate}
\item [(11a)] $\tilde{\varphi}_{e}^{\left(n\right)}\left(\left[0,\,1\right]\right)\subset\varphi_{e}^{\left(n\right)}\left(\left[0,\,1\right]\right),$
or
\item [(11b)] $0<\epsilon_{n}<\frac{g_{\left(n\right)}}{2\overline{\kappa}_{\left(n\right)}},$
\end{enumerate}
then $\tilde{\Phi}$ satisfies the hypothesis of Theorem \ref{thm:boundedwithgaps}.
\end{thm}

We wish to formulate a similar theorem for pertubed systems corresponding
to Theorem \ref{thm:unboundedwithgaps}. If we now assume that the
ESC and (\ref{eq:subexponential}) hold, then we see that 
\begin{align*}
0 & \le\frac{1}{n}\log\frac{\overline{\tilde{\kappa}}_{n}}{\underline{\tilde{\kappa}}_{n}}\le\frac{1}{n}\log\frac{\overline{\kappa}_{n}}{\underline{\kappa}_{n}}+\frac{1}{n}\log\prod_{k=1}^{n}\left(1+\overline{\gamma^{\left(k\right)}}\right)-\frac{1}{n}\log\prod_{k=1}^{n}\left(1-\overline{\gamma^{\left(k\right)}}\right)\\
 & \sim\frac{1}{n}\log\frac{\overline{\kappa}_{n}}{\underline{\kappa}_{n}}+\frac{1}{n}\sum_{k=1}^{n}\overline{\gamma^{\left(k\right)}}.
\end{align*}
From (\ref{eq:subexponential}) it suffices to have $\frac{1}{n}\sum_{k=1}^{n}\overline{\gamma^{\left(k\right)}}=0,$
which holds whenever 
\[
\lim_{k\to\infty}\overline{\gamma^{\left(k\right)}}=0.
\]

\begin{thm}
Suppose that $\Phi$ is a nonautonomous IFS of linear functions satisfying
the hypothesis of Theorem \ref{thm:unboundedwithgaps} and that $\left(\gamma_{e}^{\left(n\right)}\left(x\right)\right)_{n\in\mathbb{N},\,e\in I^{\left(n\right)}}$
is a sequence of H\"older-continuous functions from $\left[0,\,1\right]$
into $\left(-\epsilon_{n},\,\epsilon_{n}\right)$ for some $\epsilon_{n}\in\left(0,\,1\right).$
Furthermore, let $\tilde{\Phi}$ be a nonlinear perturbation of $\Phi$
defined as
\[
\tilde{\varphi}_{e}^{\left(n\right)}\left(x\right)=\varphi_{e}^{\left(n\right)}\left(u_{e}^{\left(n\right)}\right)+\left(\left(\varphi_{e}^{\left(n\right)}\right)'\right)\int_{u_{e}^{\left(n\right)}}^{x}\left(1+\gamma_{e}^{\left(n\right)}\left(t\right)\right)\,dt.
\]
 If 
\[
\frac{1}{n}\sum_{k=1}^{n}\overline{\gamma^{\left(k\right)}}\to0,
\]
in particular, if $\lim_{k\to\infty}\overline{\gamma^{\left(k\right)}}=0,$
and either
\begin{enumerate}
\item [(13a)] $\tilde{\varphi}_{e}^{\left(n\right)}\left(\Delta\right)\subset\varphi_{e}^{\left(n\right)}\left(\Delta\right),$
or
\item [(13b)] for all $n$, $\gamma_{e}^{\left(n\right)}\left(\Delta\right)\subset\left(-\epsilon_{n},\,\epsilon_{n}\right)$
and $0<\epsilon_{n}<\frac{g_{\left(n\right)}}{2\overline{\kappa}_{\left(n\right)}},$
\end{enumerate}
then $\tilde{\Phi}$ satisfies the hypothesis of Theorem \ref{thm:unboundedwithgaps}.
\end{thm}

\subsection*{Acknowledgements:}

The author would like to thank Mariusz Urba\'{n}ski, who provided
much guidance during this project. The author also acknowledges the
financial support of the Warsaw Center of Mathematics and Computer
Science, where he spent one semester working on this project under
the guidance of Krzysztof Bara\'{n}ski. This project was carried out
as part of a doctoral program which was funded by CONACYT.

\bibliographystyle{plain}
\bibliography{My_Library}

\end{document}